\documentclass[11pt,a4paper]{article} 

\usepackage[latin1]{inputenc}
\usepackage{amsmath,amssymb,amsthm,srcltx}
\usepackage{mathtools}
\usepackage{cases}
\usepackage[dvips]{color}
\usepackage[margin]{fixme}
\usepackage{enumerate}
\usepackage{stmaryrd}
\usepackage[colorinlistoftodos]{todonotes}
\usepackage{bbm}
\usepackage{dsfont}
\numberwithin{equation}{section}
\usepackage{fullpage}
\usepackage[normalem]{ulem}

\def\dbE{\mathbb{E}}
\def\dbF{\mathbb{F}}

\def\dbH{\mathbb{H}}
\def\dbI{\mathbb{I}}

\def\dbL{\mathbb{L}}

\def\dbP{\mathbb{P}}
\def\dbR{\mathbb{R}}

\def\dbZ{\mathbb{Z}}

\newcommand{\cC}{\mathcal{C}}
\newcommand{\cD}{\mathcal{D}}

\newcommand{\cF}{\mathcal{F}}

\newcommand{\cI}{\mathcal{I}}

\newcommand{\cL}{\mathcal{L}}

\newcommand{\cP}{\mathcal{P}}

\newcommand{\cR}{\mathcal{R}}

\newcommand{\cU}{\mathcal{U}}
\newcommand{\cV}{\mathcal{V}}
\newcommand{\cW}{\mathcal{W}}

\def\no{\noindent}

\def\ms{\medskip}

\def\q{\quad}

\def\pa{\partial}
\def\cd{\cdot}
\def\cds{\cdots}

\def\a{\alpha}
\def\b{\beta}
\def\d{\delta}
\def\e{\varepsilon}
\def\Om{\Omega}

\def\l{\lambda}
\def\L{\Lambda}
\def\t {\tau}

\def\th {\theta}
\def\Th{\Theta}
\def\we {\wedge}

\def\ol{\overline}

\newcommand{\ba}{\begin{array}} 
\newcommand{\ea}{\end{array}}
\newcommand{\be}{\begin{equation}}
\newcommand{\ee}{\end{equation}}
\newcommand{\bea}{\begin{eqnarray}}
\newcommand{\eea}{\end{eqnarray}}
\newcommand{\beaa}{\begin{eqnarray*}}
\newcommand{\eeaa}{\end{eqnarray*}}

\DeclareMathOperator*{\esssup}{ess\,sup}

\DeclareMathOperator*{\argmax}{arg\,max}
\DeclareMathOperator*{\argmin}{arg\,min}

\newtheorem{theorem}{Theorem}[section] 
\newtheorem{assumption}[theorem]{Assumption}

\newtheorem{corollary}[theorem]{Corollary}

\newtheorem{definition}[theorem]{Definition}
\newtheorem{lemma}[theorem]{Lemma}
\newtheorem{proposition}[theorem]{Proposition}

\theoremstyle{definition}
\newtheorem{example}[theorem]{Example}
\newtheorem{remark}[theorem]{Remark}

\begin{document}

	\renewcommand {\theequation}{\arabic{section}.\arabic{equation}}
	\def\thesection{\arabic{section}}

	\numberwithin{equation}{section}
	\numberwithin{theorem}{section}
	
	\numberwithin{figure}{section}

\title{Principal-agent problem with multiple principals}

\author{Kaitong HU\footnote{CMAP, \'Ecole Polytechnique, F-91128 Palaiseau Cedex, France, {\tt kaitong.hu@polytechnique.edu}.}
        \and Zhenjie REN\footnote{CEREMADE, Universit\'e Paris Dauphine, PSL, F-75775 Paris Cedex 16, France, {\tt ren@ceremade.dauphine.fr}.}
        \and Junjian YANG\footnote{FAM, Fakult\"at f\"ur Mathematik und Geoinformation, Vienna University of Technology, A-1040 Vienna, Austria, {\tt junjian.yang@tuwien.ac.at}.}
       }

\date{\today}

\maketitle

\begin{abstract} 
 We consider a moral hazard problem with multiple principals in a continuous-time model. The agent can only work exclusively for one principal at a given time, so faces an optimal switching problem. Using a randomized formulation and techniques from the theory of backward SDEs, we manage to represent the agent's value function and his optimal effort by an It\^o process. This representation further helps to solve the principals' problem in case we have infinite number of principals in the sense of mean field game. Finally, to justify the mean field formulation, we develop the so-called backward propagation of chaos, which may carry independent interest itself.
\end{abstract}

\vspace{-0.2cm}\newline


\section{Introduction}

The principal-agent problem is a study of optimizing the incentives, so central in economics. In particular, the optimal contracting between the two parties, principal and agent(s), is called moral hazard, when the agent's effort is not observable
by the principal. It has been widely applied in many areas of economics and finance, for example in corporate finance, see \cite{BD05} and the recent works by Cvitani\'c, Possama\"{i} and Touzi \cite{CPT17}, El Euch, Mastrolia, Rosenbaum and Touzi \cite{EMRT2021}, Cvitani\'c and Xing \cite{CX2018}, Baldacci, Manziuk, Mastrolia and Rosembaum \cite{BMMR2019}. More recently, we also witness the works A\"{i}d, Possama\"{i} and Touzi \cite{APT22}, Alasseur, Ekeland, \'Elie, Hern\'andez and Possama\"{i} \cite{AEEHP20}, \'Elie, Hubert, Mastrolia and Possama\"{i} \cite{EHMP2021}, Alasseur, Farhat and Saguan \cite{AFS2020}, using the principal-agent formulation to study how to design an optimal electricity contract and how to encourage people to embrace the energy transition. 
We would also like to mention the very recent applications related to epidemic control by Aurell, Carmona, Dayanikli and Lauri\`ere \cite{ACDL2020}, Hubert, Mastrolia, Possama\"{i} and Warin \cite{HMPW2020}. 

While the research on the discrete-time model can be dated back further, the first paper on continuous-time principal-agent problem is the seminal work by Holmstr\"om and Milgrom \cite{HM87}, who study a simple continuous-time model in which the agent gets paid at the end of a finite time interval. They show that optimal contracts are linear in aggregate output when the agent has exponential utility with a monetary cost of effort. The advantage of the continuous-time model is further explored by Sannikov \cite{Sannikov08}. Not only he considers a new model which allows the agent to retire, but also (and  more importantly in the mathematical perspective) he introduces new dynamic insights to the principal-agent problem, and it leads to simple computational procedure to find the optimal contract by solving an ordinary differential equation in his case.  

Later, the idea of Sannikov is interpreted and introduced to the mathematical finance community by Cvitani\'c, Possama\"i and Touzi \cite{CPT18}. Let us illustrate their contribution with a toy model. 
Denote by $\xi$ the contract paid at the end of a finite time interval $[0,T]$. Assume the agent faces the following optimization:
 \beaa
  \max_\a \dbE\left[\xi(X^\a) - \int_0^T c(\a_t) dt \right],\q\mbox{where}\q dX^\a_t = dW_t + \a_t dt .
 \eeaa
The crucial observation in \cite{CPT18} is that both the contract $\xi$ and the Agent's best response $\a^*[\xi]$ can be characterized by the following backward stochastic differential equation (in short, BSDE, for readers not familiar with BSDE we refer to \cite{PP90, EKPQ97}, and in particular to \cite{CZ12} for the applications on the contract theory):
  \bea\label{eq:introbsde}
    dY_t = -c^*(Z_t) dt +Z_t dW_t,\q Y_T =\xi, \q\mbox{where}\q c^*(z) = \max_a \big\{az - c(a) \big\},
  \eea
namely, $\xi = Y_T $ and $\a^*_t[Z] = \argmax_a \big\{aZ_t - c(a) \big\}$ for all $t\in [0,T]$. This induces a natural (forward) representation of the couple $(\xi, \a^*[\xi])$:
  \begin{align*}
    \begin{cases}
      \xi = Y_T^{Y_0, Z} := \displaystyle Y_ 0 -\int_0^T c^*(Z_t) dt + \int_0^T Z_t dW_t  \\ 
      \a^*_t[\xi] := \a^*_t[Z] = \displaystyle\argmax_a \big\{aZ_t - c(a)\big\}, \,\, \mbox{ for all $t\in [0,T]$}
    \end{cases}
     \mbox{ for some } (Y_0,Z),
  \end{align*}
and this neat representation transforms the once puzzling principal's problem into a classical control problem, namely,
$$
\max_\xi \dbE\Big[U\left(X^{\a^*[\xi]}_T - \xi\right)\Big] = \max_{Y_0,Z} \dbE\Big[U\left(X^{\a^*[Z]}_T - Y^{Y_0, Z}_T\right)\Big].
$$
Cvitani\'c, Possama\"i and Touzi \cite{CPT18} provide also an extension of Sannikov \cite{Sannikov08} allowing to address a wide spectrum of principal-agent problems, in particular problems with volatility control, which required the use of second-order BSDEs. We also mention that the work of Sannikov \cite{Sannikov08} has also been recently revisited by Possama\"{i} and Touzi \cite{PT2020}.
This idea of representation is further applied to study the case where the principal can hire multiple agents. See \'Elie and Possama\"i \cite{EP19}, Koo, Shim and Wang \cite{KSS2008}, Baldacci, Possama\"i and Rosenbaum \cite{BPR2021}. Eventually, in \'Elie, Mastrolia and Possama\"i \cite{EMP18} the authors follow the same machinery to study the model where the principal hires infinite number of agents using the formulation of mean field games (as for the mean field game we refer to the seminal paper \cite{LL07} and the recent books \cite{CD1} and \cite{CD2}). See also the works by Carmona and Wang \cite{CW20201}, \'Elie, Hubert, Mastrolia and Possama\"{i} \cite{EHMP2021} for a continuum of agents with mean-field interactions.

There are fewer existing literature on the model concerning one agent facing multiple principals. Meanwhile, along  the emerging  services, such as private teachings, household services, Uber-like taxis, and so on, more and more people are involved in careers during which they can switch from one employer (typically an online platform) to another more freely than ever before. This new trend evokes our curiosity to have a theoretical insight on the incentive-mechanism of such job markets.  

In the 1980s and 1990s, the economists investigated the common agency problem in the discrete time case, i.e., the agent simultaneously works on different projects for different principals. See \cite{Baron1985, BS1982, BW1985, BW1986,DGH1997}. 
Recently, in \cite{MR18} the authors the common agency problem in the continuous time case.
However, to the best of our knowledge, no one has yet considered a $n$-principal/$1$-agent model where the agent can only exclusively work for one principal at a given time.
In such a model, the agent is facing an optimal switching (among the principals) problem, i.e., the agent is looking for optimal stopping times to switch and optimal regimes. According to the classic literature of optimal switching, see e.g.~\cite{PVZ09, HZ10, HT10, CEK11}, the counterpart of the BSDE characterization \eqref{eq:introbsde} for the agent's problem would be a system of reflected BSDE in the form:
  $$
    dY_t = -f_t(Y_t, Z_t)dt +Z_tdW_t - dK_t, \q Y_T= \xi,
  $$
where $K$ is an increasing process satisfying some (backward) Skorokhod condition. The presence of the process $K$ and its constraint make it difficult to find a representation of the couple $(\xi, \a^*[\xi])$ as in \cite{CPT18}. In this paper, we propose an alternative facing this difficulty. Instead of letting the agent choose stopping times to change his employers, we study a randomized optimal switching problem where the switching time is modelled by a random time characterized by a Poisson point process and the agent influences the random time by controlling the intensity of the Poisson point process. Here, the intensity describes the hesitation of changing employer. It is fair to note that similar randomized formulations of switching have been discussed in the literature, see e.g.~\cite{Bruno09, EK10}. In such framework, we may characterize $(\xi, \a^*[\xi])$ by the solution to a system of BSDEs (without reflection). 

Unfortunately, this characterization of the agent's problem and the corresponding representation do not help us to solve simply the principals' problem. We observe that the optimizations the principals face are time-inconsistent, that is, these cannot be solved by dynamic programming. In order to get around this difficulty, we suggest the following two approaches:
\begin{itemize}
\item We propose and study a ``suboptimal'' optimization for the principals which turns out to be time-consistent. In this case, we find the PDE system characterizing the Nash equilibrium and solve it numerically.

\item We also note that this time-inconsistency disappears once the number of principals tends to infinity. In the setting of infinite number of principals, it is natural to adopt the formulation of mean field game. We prove the existence of mean field equilibrium largely based on the recipes in Lacker \cite{Lacker15}. Further, in order to justify our mean field game formulation, we introduce a machinery named ``backward propagation of chaos'' which may carry independent interest itself.
\end{itemize}

The rest of the paper is organized as follows. In Section \ref{sec:problem}, we state our moral hazard problem with $n$-principal and one agent, we shall solve the agent's problem under the previously mentioned randomized optimal switching formulation, and observe the time-inconsistency of the $n$-principal problem. In Section \ref{sec:secondbest} we study the time-consistent ``suboptimal'' optimization for the principals.  Finally in Section \ref{sec:mfg}, we shall derive the mean field game among the principals, prove its existence, and justify it using the ``backward propagation of chaos'' technique.

\section{Moral hazard problem : $n$ principals and one agent}\label{sec:problem}

 In this paper we consider the principal-agent problem on the a finite time horizon $[0,T]$ for some $T>0$. 
 The main novelty is to introduce the new model and method to allow the agent to choose working for different employers. In this section, we set up the model in which one agent switches among $n$ different principals. 
 
 For the agent, the set of possible regimes is $\mathbb{I}_n:=\{1,2,\cdots,n\}$.
 Denote by $C([0,T];\dbR^n)$ the set of continuous functions from $[0,T]$ to $\dbR^n$, endowed with the supremum norm $\|\cdot\|_T$, where 
   $$ \|x\|_t:=\sup_{0\leq s\leq t}|x_s|, \quad t\in[0,T],\quad x\in C([0,T];\dbR^n). $$
 Denote by $D([0,T],\mathbb{I}_n)$  the set of c\`adl\`ag functions from $[0,T]$ to $\mathbb{I}_n$. 
 We introduce the canonical space $\Omega:=C([0,T];\dbR^n)\times D([0,T];\mathbb{I}_n)$, denote the canonical process by $(W,I)$, and the canonical filtration by $\dbF = \{\cF_t\}_{t\in [0,T]}$. 
 The process $X_t$ represents the outputs of the principals and $I_t$ records for which principal the agent is working at time $t$. Denote by $\dbF^W$ the filtration generated by the process $W$ alone. We also define $\dbP_0$ as the Wiener measure on $C([0,T],\dbR^n)$.
 
 In our model, for simplicity, the output process follows the dynamic 
   \begin{equation} \label{eq:OutputSDE}
     X_t =  \int_0^t e_{I_s} ds + W_t, \q\mbox{where}\q  e_i = (0,\cdots,0, 1,0,\cdots, 0)^\top \in\dbR^n,
   \end{equation}
  where $W$ is a $n$-dimensional Brownian motion.
 That is, the principal employing the agent has a constant-$1$ drift while outputs of the others just follow the Brownian motion. 
 The agent controls the process $I$, that is, he switches from one employer to another so as to maximize his utility: 
   \bea \label{eq:agenttodefine}
     V_0^A(\xi,w) = \sup_{I}\dbE\left[\sum_{i=1}^n\left(\xi^i\mathbf{1}_{\{I_T=i\}} + \int_0^T \theta^i_t\mathbf{1}_{\{I_t=i\}}du\right) - C_T(I)\right],
   \eea
 where $\xi^i$ is the reward at the terminal time $T$, $\theta^i$ is the wage payed by principal $i$, and $C_T(I)$ denotes the cost of switching $I$ up to time $T$ which will be defined later when we analyze the agent's problem in detail. 
 Naturally, the payment $\xi^i$ should be an $\cF^W_T$-measurable random variable and $\theta^i$ should be an $\dbF^W$-adapted process.

Providing the contracts $\{(\xi^i, \theta^i)\}_{i\in \dbI_n}$ respectively, the n principals search among them a Nash equilibrium so as to maximize their profits:
  \bea\label{eq:principalfirst}
    V^{P,i}_0 = \max_{(\xi^i, \theta^i)} \dbE\left[X^{i,\xi^i}_T - U^i(\xi^i) \mathbf{1}_{\{I_T=i\} } -\int_0^T U^i(\theta^i_t) \mathbf{1}_{\{I_t=i\} }dt\right],
  \eea
  where $X^{i,\xi^i}$ denotes the optimal output process controlled by the agent. 
Here, we use $U^i$ to describe the principals' preference. 
We assume that $U^i$ is convex and increasing. 
We also note that $x\mapsto -U^i(-x)$ is concave and increasing, i.e., the usual definition of a utility function. See Section \ref{sec:numerical} for concretes examples.

\begin{remark}
	In the agent problem \eqref{eq:principalfirst}, the function $U^i$ applies only to the final payment $\xi^i$ and the wage $\theta^i$. 
	We may also have another function $U^i_0$ applying to $X^{i,\xi^i}_T$ as well, i.e., 
	 \begin{align} \label{eq:principalU0}
	 	V^{P,i}_0 = \max_{(\xi^i, \theta^i)} \dbE\left[-U^i_0\big(-X^{i,\xi^i}_T\big) - U^i(\xi^i) \mathbf{1}_{\{I_T=i\} } -\int_0^T U^i(\theta^i_t) \mathbf{1}_{\{I_t=i\} }dt\right].  
	 \end{align}
	For the sake of simplicity, we assume that $U^i_0(x) = x$, for which we have \eqref{eq:principalfirst}. 
	See Remark \eqref{rem:U0=id} for more details. 
\end{remark}

 As in \cite{Sannikov08, CPT18, LRTY2022}, we are going to provide a representation of the value function of the agent's problem so as to solve the principals' problem by dynamic programming.

\subsection{Agent's problem} 
 
In our model, instead of allowing the agent to control the process $I$ in a ``singular'' way, we assume that $I$ is driven by a Poisson random measure (p.r.m.) and allow the agent to control its intensity. 
More precisely, in the weak formulation of stochastic control, the agent aims at choosing his optimal control among:
 \beaa
   \Bigg\{\dbP^\a\in \cP(\Om): && \dbP^\a\mbox{-a.s.}~\mbox{$W$ is a Brownian motion,}~~dI_t = \int_{\dbI_n} (k - I_{t-})\mu(dt, dk),\\
    && \mbox{where $\mu$ is a p.r.m.~with intensity}~\a(dt, \cd) =\sum_{k\in \dbI_n} \a^k_t \d_k dt, \\
    && \mbox{for some nonnegative $\dbF$-adapted $(\a^k)_{k\in \dbI_n}$}, \,\, \mbox{$W$ and $\mu$ are independent} \Bigg\},
 \eeaa
 where $\cP(\Om)$ is the set of all probability measures on $\Om$ and $\delta_k$ is the Dirac measure.

\begin{remark}\label{rem:jumpdensity}
 Define the first jump time of the process $I$:
  \bea  \label{eq:firstjump}
    \tau_t:= \inf\{s\geq t: I_s\neq I_t\}.
  \eea
It follows from the Girsanov theorem for multivariate point processes (see e.g.~\cite{Jacod75}) that for any $\dbF^W$-adapted processes $\{\gamma(i)\}_{i\in \dbI_n}$ we have
  \beaa
    && \dbE_t^{\dbP^\a}\left[\gamma_{\t_t}(I_{\t_t}) \mathbf{1}_{\{\t_t\le s, ~I_{\t_t} =j\}}\right] = \dbE_t^{\dbP^\a}\left[\int_t^s \b^{I_t,\a}(t,u) \a_u^j \gamma_u(j) du\right],\q \mbox{for all $j\in \dbI_n$ and $j\neq I_t$},\\
    && \mbox{as well as}\q \dbP^\alpha\left[\tau_t > s \big| \cF_s \right] = \b^{I_t, \a}(t,s ),\q\mbox{where}~~\b^{i,\a}(t,s ):= \exp\bigg(-\int_t^s \sum_{j\neq i}\alpha^j_u du\bigg).
 \eeaa
These results will be useful in the upcoming calculus. 
 \end{remark}
With the intensity $\a$ introduced as above, we can now make it precise how to define the cost of switching in \eqref{eq:agenttodefine}.  Given $n$ contracts $\{(\xi^i,\theta^i)\}_{i\in \dbI_n}$, the agent solves the following optimal switching problem 
  \bea\label{eq:agentVF}
    V_0^A = \sup_{\alpha}\dbE^{\dbP^\alpha}\Bigg[\sum_{i=1}^n\xi^i\mathbf{1}_{\{I_T=i\}}
									+\int_0^T \bigg(\sum_{i=1}^n\theta^i_u\mathbf{1}_{\{I_u=i\}} 
									-\frac{1}{n-1} \sum_{i\neq I_u} c\big((n-1)\alpha^i_u\big)\bigg)du\Bigg].
  \eea
Here, $c$ is a cost function, i.e., a non-negative convex function.
Intuitively, the intensity process $\alpha$ describes the hesitation of changing employer. 
The bigger $\alpha$ is, the less hesitation the agent has to change his employer. 
 The agent, when working for the principal $i$, can choose some effort $\alpha^j$ in order to increase the probability of switching from principal $i$ to principal $j$.
The normalization $\frac{1}{n-1}c\big((n-1)\cd\big)$ is for later use, as $n\rightarrow \infty$.

\begin{remark}
 Unlike the moral hazard problems proposed by Holmstr\"om and Milgrom \cite{HM87}, Sannikov \cite{Sannikov08} and Cvitani\'c, Possama\"i and Touzi \cite{CPT18}, in our setting since the drift of the output diffusion $X$ \eqref{eq:OutputSDE} is constant, that is, the principals cannot incite the agent to make more ``effort'' during his work. Instead, the principals make contracts based on the information of $W=(W^1, \cds, W^n)$ to give the agent the incentive to switch among them (see Remark \ref{rem:incentive}). Note that the principals still do not have insight to the intensity that the agent chooses to switch among the different job opportunities. Therefore, the principals encounter a new type of moral hazard problem, and need to provide incentives so as to obtain the most advantageous feedback from the agent. 
\end{remark} 

As in the classical literature of the optimal switching problems, we shall use the dynamic programming principle to obtain the system of equations characterizing the value function. First, define the dynamic version of \eqref{eq:agentVF}:
 \begin{equation*}
    V_t^A = \esssup_{\alpha}\dbE_t^{\dbP^\alpha}\Bigg[\sum_{i=1}^n\xi^i\mathbf{1}_{\{I_T=i\}} 
                                                             + \int_t^T \bigg(\sum_{i=1}^n \theta^i_u\mathbf{1}_{\{I_u=i\}}-\frac{1}{n-1} \sum_{i\neq I_u} c\big((n-1)\alpha^i_u\big)\bigg)du\Bigg].
 \end{equation*}	
Recall $\t_t$ defined in \eqref{eq:firstjump}. By the dynamic programming, see, e.g., \cite[Theorem 5.2]{FM20}, we have
 \begin{multline*}
    V_t^A = \esssup_{\alpha}\dbE_t^{\dbP^\alpha}\Bigg[\sum_{i=1}^n\xi^i\mathbf{1}_{\{I_T=i,~\t_t >T\}} + V^A_{\t_t}\mathbf{1}_{\{\t_t\le T\}}  \\
    									+\int_t^{\t_t\we T} \bigg(\sum_{i=1}^n \theta^i_u\mathbf{1}_{\{I_u=i\}} 
									- \frac{1}{n-1} \sum_{i\neq I_u} c\big((n-1)\alpha^i_u\big)\bigg)du \Bigg].
 \end{multline*}
Further, by defining $V^{A, i}_t := V^A_t |_{I_t =i}$, the value of the control problem at time $t$ given $I_t = i$, we obtain
 \begin{align}
   V_t^{A,i} &= \esssup_{\alpha}\dbE_t^{\dbP^\alpha}\Bigg[\xi^i\mathbf{1}_{\{\tau_t>T\}} 
							+ \sum_{j\neq i} \bigg(V^{A,j}_{\tau_t}\mathbf{1}_{\{\t_t\le T, ~ I_{\t_t} =j\}}
							+ \int_t^{\tau_t\we T} \frac{\theta^i_u - c\big((n-1)\a^j_u\big)}{n-1}du \bigg)\Bigg] \notag\\
             &= \esssup_{\alpha}\dbE_t^{\dbP^\alpha}\Bigg[\xi^i\beta^{i,\alpha}(t,T) + \int_t^T\beta^{i,\alpha}(t,u)\sum_{j\neq i}\bigg(\alpha_u^{j}V^{A,j}_u + \frac{\theta^i_u-c\big((n-1)\a^j_u\big)}{n-1}\bigg)du \Bigg]. \label{eq:agentclassic}
 \end{align}

The second equality above is due to the results in Remark \ref{rem:jumpdensity}. In view of \eqref{eq:agentclassic}, it becomes a classical stochastic control problem. In particular, the value function of this control problem can be characterized by the BSDE. 

\begin{assumption} \label{assum:c}
Assume that the cost function $c$ is convex, and $c$ takes value of $+\infty$ out of a compact set $K$. We also assume that there exists a unique maximizer 
  \begin{align*}
    a^*(y) := \argmax_{a\ge 0}\big\{a y - c(a) \big\} \in K,\q\mbox{for all }\,\, y\in \dbR.
  \end{align*}
Define the convex conjugate of $c$
  \begin{align*}
  	 c^*(y) := \sup_{a\ge 0} \big\{a y - c(a) \big\},\q\mbox{for all }\,\, y\in \dbR.
  \end{align*}
Further assume that $c^*$ is Lipschitz continuous. 
\end{assumption}

\begin{example}\label{eg:costhesistation}
As an example, consider $c(a):=\frac{1}{2}a^2$ on $K:=[0,\overline{K}]$ for some $\overline{K}>0$. 
Then, 
 \begin{align}\label{eq:acstar}
   a^*(y)=\begin{cases}
            0, & y\leq 0, \\
            y, & 0\leq y\leq \overline{K}, \\
            \overline{K}, & y\geq \overline{K},
          \end{cases}
    \quad \mbox{and} \quad
   c^*(y)=\begin{cases}
           0, & y\leq 0, \\
           \frac{1}{2}y^2, & 0\leq y\leq \overline{K}, \\
           \overline{K}y-\frac{1}{2}\overline{K}^2, & y\geq \overline{K}.
          \end{cases}
 \end{align}
 Obviously, the function $c^*$ is Lipschitz continuous with constant $\overline{K}$.
 \end{example}

\begin{proposition} \label{prop:nagent}
 Under Assumption \ref{assum:c}, given $\xi^i \in\dbL^2(\dbP_0), \theta^i\in \dbH^2(\dbP_0)$\footnote{We denote by $\dbH^2(\dbP_0)$ the space of adapted processes $w$ such that 
             $$ \|\theta\|^2_{\dbH^2(\dbP_0)}:= \dbE^{\dbP_0}\left[\left(\int_0^T|\theta_u|^2du\right)^{\frac{1}{2}}\right] <\infty. $$
             } for all $i\in \dbI_n$,  the following system of BSDEs has a unique solution $(Y^i, Z^i)_{i\in\dbI_n}:$
   \bea\label{eq:BSDEnagent}
     Y_t^i = \xi^i + \int_t^T \bigg(\frac{1}{n-1}\sum_{j\neq i}c^*(Y^j_u - Y^i_u) + \theta^i_u \bigg)du - \int_t^T Z^i_u\cdot  d W_u,~~ i\in \dbI_n,~~\mbox{$\dbP_0$-a.s.}
   \eea
 Moreover, we have $Y^i_t = V^{A, i}_t$, $\dbP_0$-a.s., i.e., the solution $Y^i_t$ of the BSDE represents the continuation utility of the agent when working for the principal $i$ at time $t$. In particular, the optimal intensity satisfies
   \bea\label{eq:optimalintensity}
     \a^{j,*}_t = \frac{1}{n-1}a^*\big(Y^j_t-Y^{I_t}_t\big), \q j\neq I_t, ~~\mbox{for all}~~t\in [0,T], \q\dbP_0\mbox{-a.s.}
   \eea
\end{proposition}

\begin{proof}
In view of the control problem \eqref{eq:agentclassic}, following the result of \cite{HL95, EKPQ97}, see also \cite{Zhang2017BSDEBook}, the corresponding BSDE reads
 \begin{align*}
   Y_t^i = \xi^i + \int_t^T\bigg(\frac{1}{n-1}\sum_{j\neq i}\sup_{a^j\geq 0}\Big\{(n-1)a^j\big(Y_u^j-Y_u^i\big) - c\big((n-1)a^j\big) \Big\}+\theta^i_u \bigg)du - \int_t^T Z^i_u\cdot d W_u.
 \end{align*}
 Then, \eqref{eq:BSDEnagent} follows from the definition of $c^*$. Since all the coefficients are Lipschitz continuous, the wellposedness of the BSDE system and the verification for the control problem is classical, see e.g.~\cite{PP90}. 
\end{proof}

\begin{remark}  \label{rem:incentive}
    We recall that $Y^i$ is the continuation value of the agent when working for the principal $i$ at time $t$. 
    The switching of the agent, determined by the optimal intensity $\frac{1}{n-1}a^*(Y^j_t-Y^{I_t}_t)$,  is  influenced by principals through their contracts. 
\end{remark}

\begin{remark}
It is noteworthy that the agent problem can be easily solved and enjoys a BSDE representation, thanks to the assumption that the contracts $(\xi^i)_i$ are $\cF^W_T$-measurable. It is crucial that the contracts do not depend on the process $I$, in other words, the principals have no right to design contracts based on the history of employment of the agent. It may contradicts the reality in some applications, but is important for our mathematical reasoning. 
\end{remark}

\begin{remark}
	As in \cite{LRTY2022}, the results could also be extended to a more general case, e.g., with discount factors or risk-aversion in the agent's utility function.
	With the presence of discount factors, we have another $-rY^i$ term in \eqref{eq:BSDEnagent}.  
	Given an invertible utility function $U_A$, substitute $\xi^i$ by $U_A^{-1}(\xi)$ in \eqref{eq:BSDEnagent}. 
	We need to make the integrability assumptions on $U_A(\xi^i)$ instead of $\xi^i$. 
\end{remark}

Before we continue, let us summarize our model:
\begin{itemize}
	\item The $i$-th principal, $i\in\{1,\cdots,n\}$, observes only $W$ and $X^i$, and in particular, she does not observe for which principal the agent is working, if he is not working for her. 
	\item During the contracting period, the agent can only work for one principal at a time, but can make an effort to increase the intensity of switching to another principal. 
	\item At any time $t\in[0,T]$, the agent receives a continuous payment from the principal he works for, i.e., $\theta^i_t$ if he works for the $i$-th principal. 
	\item At the time $T$, the agent receives a terminal payment from the principal he works for at that time, i.e., $\xi^i$ if he ends up working for the $i$-th principal. 
	\item Each principal $i$ designs a contract, i.e., a couple $(\xi^i,\theta^i)$, in order to compensate the agent for his work but most importantly to incentivize the agent to continue working for her, or to come to work for her if he is actually working for another principal $j\neq i$. 
\end{itemize} 
 
\subsection{Principals' problem: time inconsistency}

In the previous section, we managed to represent the value function of the agent by an It\^o process (Proposition \ref{prop:nagent}). 
As in \cite{Sannikov08, CPT18}, we expect that this representation would help us to solve the principals' problem by the dynamic programming approach. 
However, in this model, this approach does not work. We shall explain in the case $n=2$ for the simplification of notation.
 
Consider the set of all contracts 
   \begin{align*}
   	  \Xi: = \Big\{ \big\{(\xi^i, \theta^i)\big\}_{i=1,2}:~ \xi^i\in \dbL^2(\dbP_0),~ \theta^i\in \dbH^2(\dbP_0)~\mbox{and}~ V^A(\xi,\theta)\ge R\Big\},
   \end{align*}
  where $R$ is the reservation value of the agent for whom only the contracts such that $V^A(\xi,\theta)\ge R$ are acceptable. Now define 
  \begin{align*}
  	 \cV := \left\{\big\{(Y^i_0, Z^i)\big\}_{i=1,2}:~ Y_0^i\ge R, ~Z^i\in \dbH^2(\dbP_0),~i=1,2\right\}.
  \end{align*}
It follows from Proposition \ref{prop:nagent} that
  \begin{align}\label{eq:contractset}
    \Xi = \Big\{\big\{(\xi^i, \theta^i)\big\}_{i=1,2}: &~~ \theta^i\in \dbH^2(\dbP_0) \mbox{ and }  \xi^i=Y^{i,Y_0^{i},Z^i,\theta^i}_T, \mbox{ where } Y^{i,Y_0^{i},Z^i,\theta^i} \mbox{ satisfies} \notag\\
	  &~~ Y^i_T =  Y^i_0 - \int_0^T \left(c^*\big(Y^j_t-Y^i_t\big) + \theta^i_t \right)dt + \int_0^T Z^i_t\cdot dW_t, ~\dbP_0\mbox{-a.s.},\notag \\
	  &~~ \mbox{with }~j\neq i ~\mbox{ and }~ \big\{(Y^i_0, Z^i)\big\}_{i=1,2}\in \cV \Big\}.
  \end{align} 
For simplicity of notation, if there is no ambiguity, we write simply $Y^i$ instead of $Y^{i,Y_0^{i},Z^i,\theta^i}$. But just keep in mind that the process $Y^i$ is controlled by $Z^i$ and $\theta^i$ with initial value $Y^i_0$. 
The corresponding optimal intensity reads $\alpha^*_t = \big(a^*(Y_t^1-Y_t^2)\mathbf{1}_{\{I_t = 2\}}, a^*(Y_t^2-Y_t^1)\mathbf{1}_{\{I_t = 1\}}\big)$.
Therefore, the principals are searching for a Nash equilibrium so as to maximize:
  \begin{align*}
    V^{P,i}_0 &= \sup_{(\xi^i, \theta^i)} \dbE^{\dbP^{\alpha^*}}\left[ X^i_T - U^i(\xi^i) \mathbf{1}_{\{I_T=i\}}  -\int_0^T U^i(\theta^i_t) \mathbf{1}_{\{I_t=i\} }dt \right] \\
              &= \sup_{(Y_0^i, Z^i, \theta^i)} \dbE^{\dbP^{\alpha^*}}\left[- U^i(Y^{i}_T) \mathbf{1}_{\{I_T=i\}} -\int_0^T \big( U^i(\theta^i_t)-1\big) \mathbf{1}_{\{I_t=i\} }dt\right] = \sup_{Y^i_0\geq R}J^i_0(Y^i_0),  
  \end{align*}
  where 
   $$ J^i_0(Y_0^i):= \sup_{(Z^i,\theta^i)}\dbE^{\dbP^{\alpha^*}}\left[  - U^i\big(Y^{i}_T\big)\mathbf{1}_{\{I_T=i\}} - \int_0^T \big(U^i(\theta^i_t) -1\big) \mathbf{1}_{\{I_t=i\}}dt\right]. $$
 
\begin{remark}\label{rem:control of principal}
Regardless the running payment $\theta^i$, the principal $i$ can only modify the contract by choosing $Y_0^i$ and $Z^i$. Note that $Y_0^i$ represents the ``expectation'' of the contract, while $Z^i$ describes its variance (in other word, the risk). 
\end{remark}
   
\begin{remark} \label{rem:U0=id} 
   If we consider the principal problem \eqref{eq:principalU0} with $U_0^i$, we obtain by It\^o's formula that
   \begin{align*}
   	 V^{P,i}_0 = \sup_{(Y_0^i, Z^i, \theta^i)}\dbE^{\dbP^{\alpha^*}}\bigg[ &- U^i(Y^{i}_T) \mathbf{1}_{\{I_T=i\}} \\ 
   	     &-\int_0^T \Big(\big( U^i(\theta^i_t) - (U_0^i)'(-X_t^i)\big) \mathbf{1}_{\{I_t=i\}} + \frac{1}{2}(U^i_0)''(-X_t^i)\Big)dt\bigg].
   \end{align*}
   Therefore, we have the additional $X$-dependency in the formula. For the sake of simplicity, we only consider the case $U_0(x) = x$.  
\end{remark}

Up to now, we are applying the same strategy as in \cite{CPT18}. Assume the control problem is time-consistent, that is, admits dynamic programming. 
Fix $Y_0$ and define the dynamic version of $J^i_0$:
   $$ J^i_t(Y^{i}_t):= \esssup_{(Z^i,\theta^i)}\dbE_t^{\dbP^{\alpha^*}}\left[ - U^i(Y^{i}_T)\mathbf{1}_{\{I_T=i\}} - \int_t^T \big(U^i(\theta^i_u) -1\big)\mathbf{1}_{\{I_u=i\}}du\right]. $$
Defining $J^{i,1}_t:= J^{i}_t(Y^{i}_t)|_{I_t =1} $ and $J^{i,2}_t = J^{i}_t(Y^{i}_t)|_{I_t =2} $ according to the different regimes, we expect to have for $i=1$
 \begin{align} \label{eq:principalR1}
   J^{1,1}_t &= \esssup_{(Z^1, \theta^1)}\dbE^{\dbP^{\a^*}}_t\left[ - U^i(Y^{1}_T)\mathbf{1}_{\{\t_t >T\}} + J_{\tau_t}^{1,2}\mathbf{1}_{\{\t_t \le T\}} - \int_t^{\tau_t\wedge T} \big(U^i(\theta^1_u) -1 \big)du \right] \nonumber \\
             &= \esssup_{(Z^1, \theta^1)}\dbE^{\dbP^{\a^*}}_t\bigg[ - \b^{1,*}(t,T) U^i(Y^{1}_T)\nonumber \\
             & \hspace{40mm} + \int_t^T\beta^{1,*}(t,u) \big( a^*(Y_u^{2}-Y^{1}_u)J^{1,2}_u + 1- U^i(\theta^1_u)\big)du\bigg],
 \end{align}
 where 
  $$   \b^{1,*}(t,s):= \exp\left(-\int_t^s a^*(Y_u^{2}-Y^{1}_u)du\right), $$
 and similarly 
 \begin{align}  \label{eq:principalR2}
   J_t^{1,2} &= \esssup_{(Z^1, \theta^1)}\dbE^{\dbP^{\a^*}}_t\left[  J^{1,1}_{\tau_t}\mathbf{1}_{\{\t_t\leq T\}}\right]  \nonumber \\
             &= \esssup_{(Z^1, \theta^1)}\dbE^{\dbP^{\a^*}}_t\left[ \int^T_{t} \b^{2,*}(t,u) a^*(Y^1_u-Y^2_u) J^{1,1}_u du\right]
 \end{align}
 where
  $$   \b^{2,*}(t,s):= \exp\left(-\int_t^s  a^*(Y^1_t -Y^2_t) du\right). $$
Although it seems promising to solve the system of value functions $(J^{i, 1}, J^{i, 2})$ as in the agent problem, note that the admissible controls $(Z^i, \theta^i)$ are constrained to be $\dbF^X$-adapted, in particular, $(Z^i, \theta^i)$ cannot depend on the regimes. 
Therefore, if  the control problem is time-consistent, then we expect to find the same process $(Z^1, \theta^1)$ solve both the optimizations \eqref{eq:principalR1} and \eqref{eq:principalR2}, which is absurd. 
It is a contradiction to the  time-consistency assumption. 

\begin{remark}
Besides, the approach for solving the $n$-principal problem has another drawback. Note that each contract $\xi^i$ for $i\in\dbI_n$ is a function of the entire vector $W$, that is, the principal needs to know the other companies' noise in order to design his own contract. This  problem can be avoided once we consider an infinite number of principals (see Section \ref{sec:mfg}), where $\xi^i$ will be a function of $W^i$.
\end{remark}

In the rest of the paper, we propose two ways to address the time-inconsistency above:
  \begin{itemize}
    \item introduce a time-consistent ``suboptimal'' solution for the principals;
    \item consider the case with infinite number of principals using mean-field formulation.
  \end{itemize} 

We refer to the recent work of Hern\'andez and Possama\"i \cite{HP2021}. They developed a dynamic programming principle for a time-inconsistent agent, and their result is applied to principal-agent model with moral hazard in the thesis of Hern\'andez \cite{Hernandez2021}.

\section{A time-consistent ``suboptimal'' solution for the principals} \label{sec:secondbest} 

\subsection{PDE characterization}

In this section we assume that the principals only care about maximizing their profits during the period when the agent is hired, i.e., the principal $i$, instead of solving both optimization problems \eqref{eq:principalR1} and \eqref{eq:principalR2}, only focus on the former.
More precisely, we consider the admissible contracts in \eqref{eq:contractset} represented by $\dbZ^i:=(Y_0^i , Z^i, \theta^i)$. The principals now aim at searching for $\widehat\dbZ^i:=(\widehat Y_0^{i} , \widehat Z^{i}, \widehat \theta^{i})$ such that
\bea \label{secondbest_eq}
\widehat\dbZ^i \in \argmax_{\dbZ^i} J^{i,i}_0 \left(\dbZ^i; J^{i,j}(\widehat \dbZ^i; \dbZ^j ); \dbZ^j \right)
\eea
where for $i\neq j$
\begin{align}
 J^{i,i}_t(\dbZ^i; J^{i,j}; \dbZ^j) &=  \dbE^{\dbP^{\a^*}}_t\left[ - U^i(  Y^{i}_T)\mathbf{1}_{\{\t_t >T\}} + J_{\tau_t}^{ i,j}\mathbf{1}_{\{\t_t \le T\}} - \int_t^{\tau_t\wedge T} \big(U^i(\theta^i_u) -1 \big)du \right]  \notag\\
             &= \dbE^{\dbP^{\a^*}}_t\bigg[ - \b^{i,*}(t,T) U^i( Y^{i}_T)   \notag\\
             &\q\q\q\q + \int_t^T\beta^{i,*}(t,u)\big(a^*( Y_u^{j}- Y^{i}_u)J^{i,j}_u +1 - U^i(\theta^i_u)\big)du\bigg] \label{eq:Jii}\\
J^{i,j}_t(\dbZ^i; J^{i,i}; \dbZ^j) &= \dbE^{\dbP^{\a^*}}_t\left[ J^{i,i}_{\tau_t}\mathbf{1}_{\{\t_t\leq T\}}\right]  
	  =  \dbE^{\dbP^{\a^*}}_t\left[ \int^T_{{t}} \b^{j,*}(t,u) a^*( Y^i_u- Y^j_u)J^{i,i}_u du\right],\label{eq:Jij}
\end{align}
and $J^{i,j}(\dbZ^i; \dbZ^j)$ is the solution to the linear system \eqref{eq:Jii}, \eqref{eq:Jij}. 

The Nash equilibrium of the modified game can be characterized by a highly coupled PDE system. In order to weaken the coupling among the two principals and facilitate the computation of equilibrium, we require the further constraint:
\bea\label{eq:specialZ}
Z^i = (Z^i_1, Z^i_2), \q \mbox{with}\q 
Z^i_j = 0\q\mbox{for}~ i\neq j ,\q
\mbox{and}\q 
|Z^i_i |\leq M \q\mbox{for some}~ M>0.
\eea
In order to characterize the Nash equilibrium, we consider the following PDE system:
\bea\label{eq:PDEsys}
\begin{cases}
\pa_t v^{i,i} -  \min_{\th^i} \left\{ \th^i  \pa_i v^{i,i}+ U^i(\th^i) \right\}
+ \frac12 M^2 \big(\pa^2_{ii} v^{i,i}\big)^+  -c^*(y^j-y^i) \pa_i v^{i,i} \\
\hspace{9.25mm} -  \left(c^*(y^i-y^j)+ \th^{j,*} \right) \pa_j v^{i,i}
+\frac12 |z^{j,*}| ^2 \pa^2_{jj} v^{i,i} + \a^*(y^j-y^i) (v^{i,j} -v^{i,i}) +1=0\\

\pa_t v^{i,j}  - \left(c^*(y^j-y^i)+ \th^{i,*} \right) \pa_i v^{i,j} -  \left(c^*(y^i-y^j)+ \th^{j,*} \right) \pa_j v^{i,j}\\
\hspace{9.5mm} + \frac12 |z^{i,*}| ^2 \pa^2_{ii} v^{i,j} +\frac12 |z^{j,*}| ^2 \pa^2_{jj} v^{i,j} 
+ \a^*(y^i -y^j) (v^{i,i} -v^{i,j}) =0\\

v^{i,i}_T= - U^i(y^i), \q\q v^{i,j}_T = 0,
\end{cases}
\eea
where $\th^{i,*} = \argmin_{\th^i} \left( \th^i  \pa_i v^{i,i}+ U^i(\th^i) \right)$ and $z^{i,*} = \argmax_{|z^i|\le M} |z^i|^2 \pa^2_{ii} v^{i,i} =  M 1_{\{\pa^2_{ii} v^{i,i} \ge 0 \}}$.
Further the Nash equilibrium $(\widehat\dbZ^i)_{i=1, \cdots, n}$ in \eqref{secondbest_eq} should satisfy
\[\widehat \theta^i_t = \th^{i,*}(t, Y_t), \q \widehat Z^i_t = z^{i,*}(t, Y_t), \q \mbox{and}\q Y^i_0 \in \argmax_{y^i} v^{ii}(0, y). \]

\begin{remark}
Due to the observation in the following Section \ref{sec:heuristic} that given a large number of principals, there is a tiny  possibility for the agent to return to the same principal after leaving her company.  The alternative equilibrium \eqref{secondbest_eq} would make more sense provided that there are many principals in the game. Also one may find the reasoning for the constraint $Z^i_j = 0$ for $i\neq j$ in the mean-field analysis in Section \ref{sec:backwardchaos}. 

It is noteworthy that even for this simplified ``suboptimal'' equilibrium, the corresponding PDE system similar to \eqref{eq:PDEsys} is non-trivial to solve. In general, for the case with $n$ principals, one need to solve the PDE system with  $n^2$ coupled PDE's and of $n$ variables. Apparently, the conventional PDE solvers cannot handle the task when $n$ is large. 
In the following section we will numerically solve it in case there are only two principals using the explicite Euler scheme. 
\end{remark}
 
\subsection{Numerical tests}  \label{sec:numerical}

In this section, we numerically solve the PDE system \eqref{eq:PDEsys} for the game involving two principals, using the explicit Euler scheme. In order to further simply the equations, we assume that the principals do not pay the salaries $(\th_t)$ but only the final compensation $\xi=Y_T$. 

In aspect of the coefficients, we adopt the quadratic cost for the hesitation in Example \ref{eg:costhesistation}, namely,
\[c(a) = \frac12 a^2 \q\mbox{on}\q [0, \ol K]\]
 and set the constant $\ol K =1$,   choose the bound of the volatility in \eqref{eq:specialZ} to be $M =1$, and consider the time horizon $T=0.5$. As for the utility functions of the principals, we set 
\beaa
U^1(y) = e^{y} -1  \q\mbox{and}\q 
U^2(y) = \big(e^{0.1 y} -1\big) / 0.1,
\eeaa
so that Principal $1$ is more risk-averse than Principal $2$. 

Recall the function $c^*$ and $a^*$ in \eqref{eq:acstar}. By solving the PDE system  \eqref{eq:PDEsys}, we obtain the (sub-)optimal contract
\bea\label{eq:suboptimal_contract}
 \xi^i =  Y^i_0 - \int_0^T  c^*\big(Y^j_t-Y^i_t\big) dt + \int_0^T M 1_{\{\pa^2_{ii} v^{i,i}(t, Y^1_t, Y^2_t) \ge 0 \}}  dW_t,
\q\mbox{for $i \neq j \in \{1,2\}$}. 
\eea
Meanwhile, the optimal intensity at time $t$ for the agent to switch from Principal $i$ to Principal $j$ is equal to
\[a^*(Y^j_t-Y^i_t).\]

As a benchmark, we shall look into the situation where there is only one principal in the market. Following the same analysis as before, the principal chooses among the  contracts of the form
\beaa
\xi = Y_0 + \int_0^T Z_t dW_t.
\eeaa
Since  the principal is  assumed to be risk-averse, the optimal choice is clearly $Z\equiv 0$, and the value function of the principal  reads
\beaa
v(y) = -U(y) + 1 * T,
\eeaa
where we recall $1$ is the drift of the output diffusion $X$. Moreover, since $U$ is strictly increasing, the principal would choose $Y_0$ equal to the reservation value $R$ of the agent in order to maximizer her profit.  

\begin{figure}
	\centering
	\includegraphics[width=16cm]{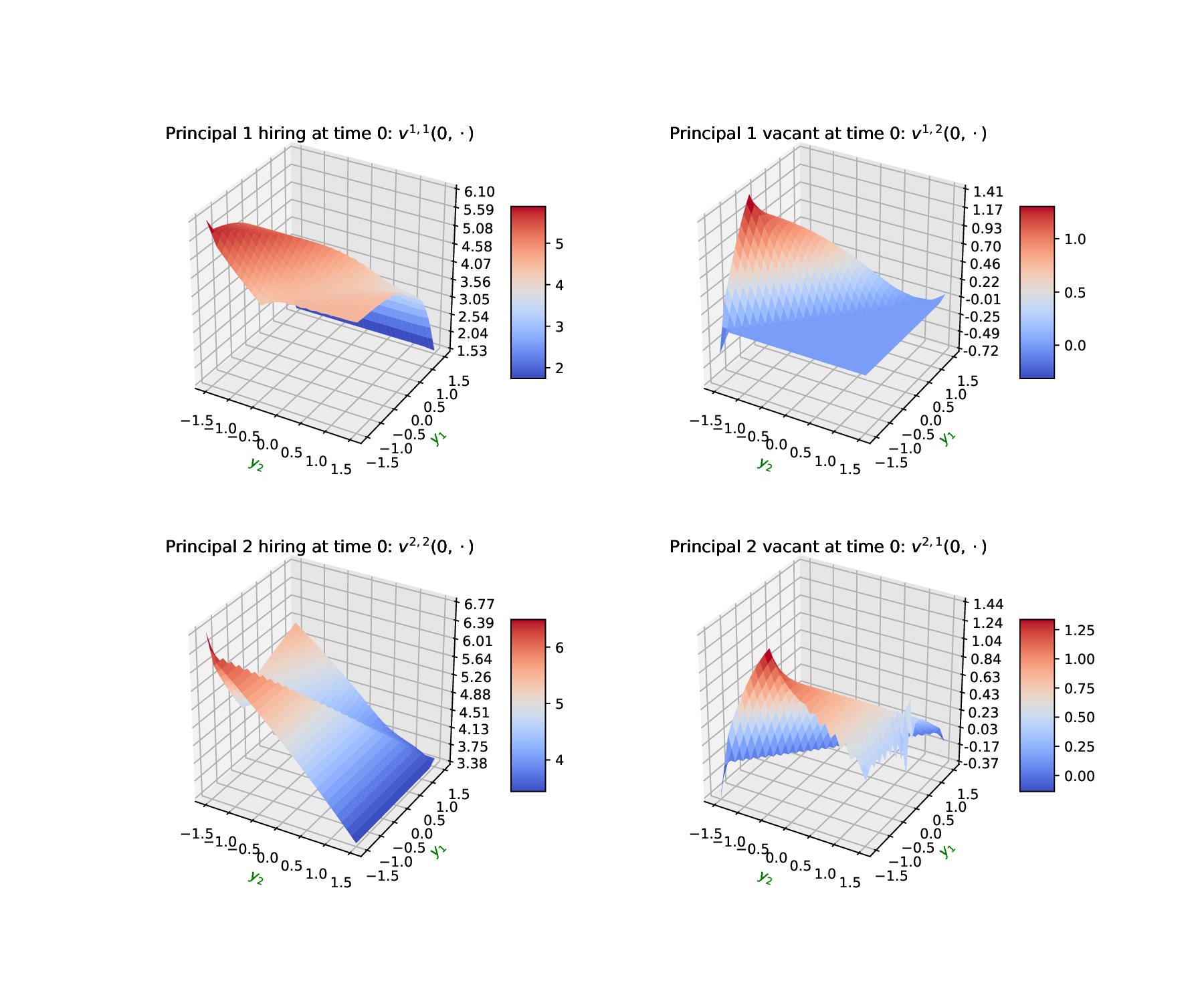}
	\caption{Value functions at time $0$}\label{fig:3D}
\end{figure}

In Figure \ref{fig:3D} we can view the numerical solution to the value functions of the two principals at time $0$.
The four graphes show the values of functions $v^{1,1}(0, \cd), ~ v^{1,2}(0, \cd), ~ v^{2,2}(0, \cd), ~ v^{2,1}(0, \cd)$ respectively. Already we can see that the value functions $y^1\mapsto v^{1,1}(0, y^1, \cd),~ y^2\mapsto v^{2,2}(0, \cd, y^2) $ are no longer concave. As a result the optimal choice of $Z^{i, *}=M 1_{\{\pa^2_{ii} v^{i,i} \ge 0 \}}$ is no longer $0$ (the optimal choice for a single principal), in other word, the competition does change the strategies of the principals. 
In order to make this point more clear, we simulate the continuation utility values $(Y^1_t)_{t\in [0,T]}$ and $(Y^2_t)_{t\in [0,T]}$ of the agent under the contacts \eqref{eq:suboptimal_contract}, as well as his decision of switching.  By setting $Y^1_0 = 0.3, Y^2_0 = -0.5$ and letting the agent start being hired by Principal $2$, we obtain the simulation result in Figure \ref{fig:simu}. As we see, at the beginning Principal $2$ sets the volatility in the contract $Z^{2,*}$ to be non-zero as an attempt to possibly close the gap between $Y^1$ and $Y^2$ so as to keep the agent working for her and later she sets the volatility $Z^{2,*}=0$ because she would like to allow the agent to leave in order to avoid the terminal payment.  Meanwhile since $Y^1$ is always bigger than $Y^2$, the agent is always interested in switching to Principal $1$ (in other word, the intensity $a^*(Y^1_t-Y^2_t)>0$) and eventually he manages it in this simulation. 

\begin{figure}
  \centering
  \includegraphics[width=15cm]{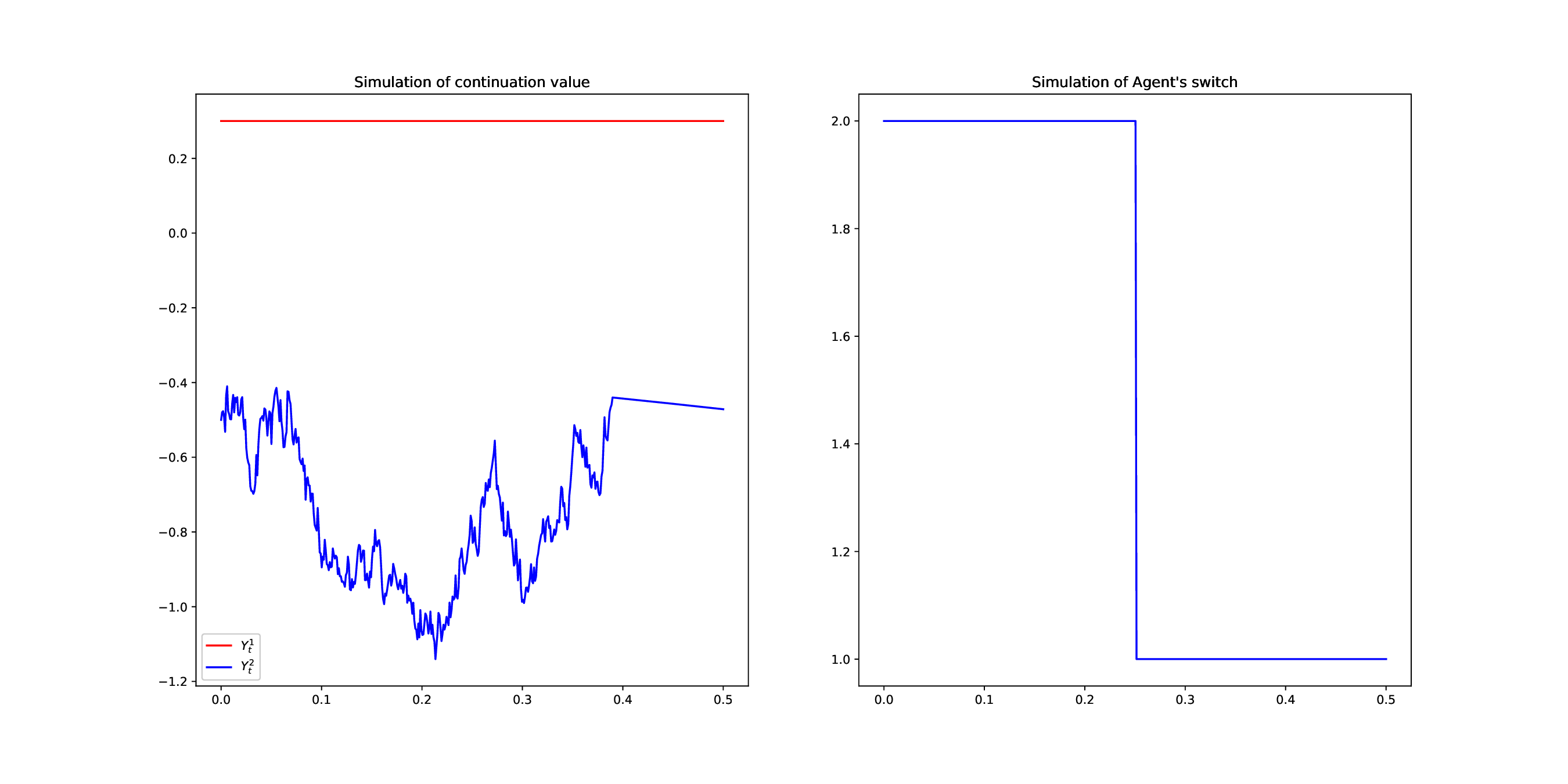}
  \caption{Simulation of the continuation values and the agent's decision}\label{fig:simu}
\end{figure}

Further in Figure \ref{fig:2D}, we observe how the position  $y^j$ influences the value function $v^{ii}$ for $i\neq j$.
We draw in different colors the value functions given different values of $y^j$, as well as in black the value function of the single principal without competition. Here are some   discoveries:
 
\begin{figure}
  \centering
  \includegraphics[width=14cm]{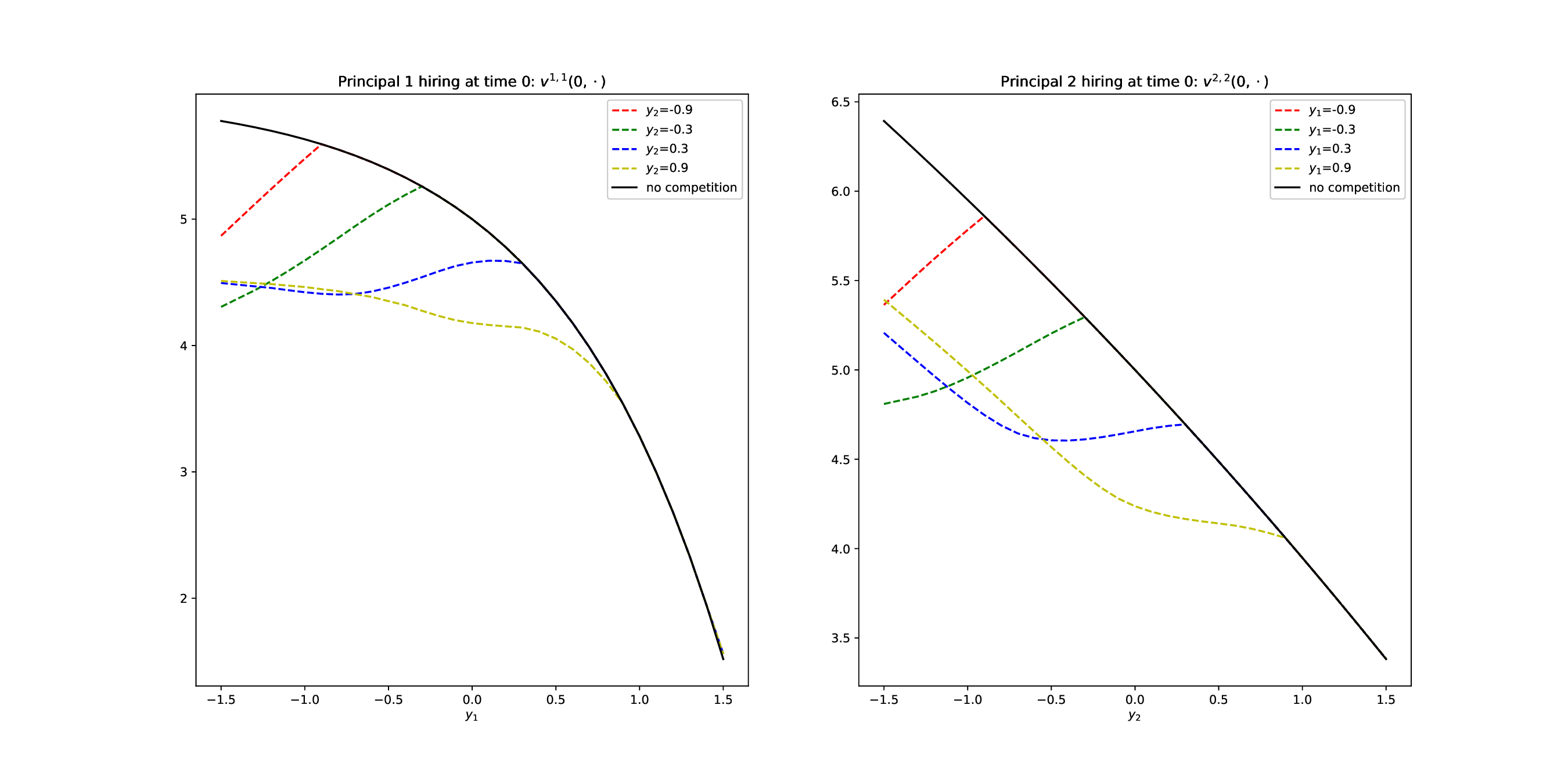}
  \caption{Value functions at time $0$} \label{fig:2D}
\end{figure}

\begin{figure}
	\centering
	\includegraphics[width=10cm]{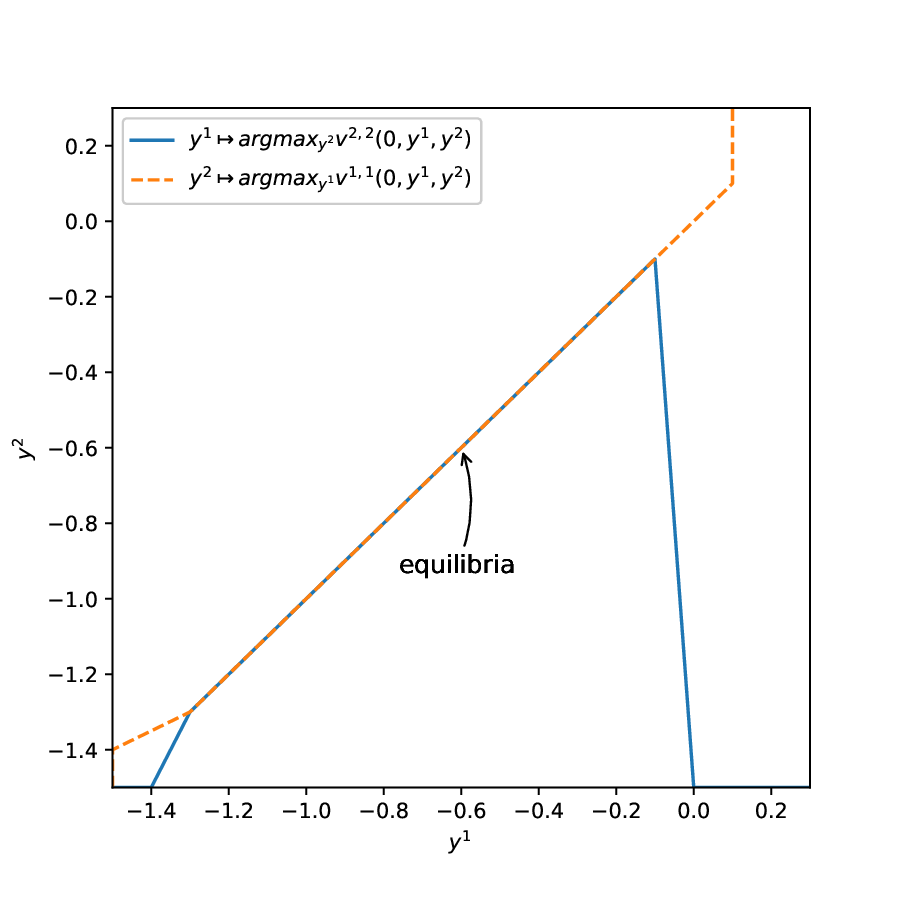}
	\caption{Equilibria at time $0$} \label{fig:equilibria}
\end{figure}

\begin{itemize}
   \item Most significantly, the value function $v^{i,i}$ coincides with the value function without competition, once the value of $y^i$ surpasses that of $y^j$;
   \item The value function $y^i\mapsto v^{i,i}(0, y^i, \cd)$ is convex only when $y^i<y^j$ and the spread is not large, in particular, it is the  scenario where the principals are willing to set the volatility to be non-zero as an attempt to keep the agent working for them, and the more risk-averse principal appears to be more reluctant to submit this risk; 
   \item Recall that the principals also choose $(Y_0^i)_{i=1,2}$ to reach an equilibrium such that
           \begin{align*}
        	 Y_0^1 = \argmax_{y^1} v^{1,1}(0,y^1, Y_0^2)
        	  \q\mbox{and}\q
          	 Y_0^2 = \argmax_{y^2} v^{2,2}(0,Y_0^1, y^2).
           \end{align*}
         As shown in Figure \ref{fig:equilibria}, we find the equilibria $(Y_0^i)_{i=1,2}$ on the domain $[-1.5,1.5]\times[-1.5,1.5] $. Curiously,  in this example the equilibria  are all on the diagonal $Y^1_0=Y^2_0$. Taking into account that the value functions $v^{1,1}, v^{2,2}$ seem locally concave on the diagonal (see Figure \ref{fig:3D}), we shall have $Z^{1,*} = Z^{2,*} = 0$ there. Since in this example $c^*(0)=0$, eventually we should have
		   \begin{align*}
		     Y^1_T &=  Y^1_0 -  \int_0^Tc^*\big(Y^2_t-Y^1_t\big)  dt + \int_0^T Z^{1,*} _t\cdot dW_t = Y^1_0,\\
		     Y^2_T &=  Y^2_0 -  \int_0^Tc^*\big(Y^1_t-Y^2_t\big)  dt + \int_0^T Z^{2,*} _t\cdot dW_t = Y^2_0,
		   \end{align*}
		that is, both principals would offer constant rewards as in the single principal case.  However, the constant reward $Y^1_0=Y^2_0$ may be still higher than the reservation value of the agent.
\end{itemize}

\section{Infinite number of principals: mean field approach} \label{sec:mfg}
 
\subsection{Heuristic analysis}\label{sec:heuristic}

Heuristically, as $n\rightarrow \infty$ the equation \eqref{eq:BSDEnagent}, which characterizes the agent's value function, converges to
  \begin{equation} \label{eq:limitY}
    dY^i_t = -\left( \int c^*(y-Y^i_t) p_t(dy) +\theta^i_t  \right)dt + Z^i_t dW^i_t,  \q  p_t =\cL(Y^i_t).
  \end{equation}
Besides, it follows from \eqref{eq:optimalintensity} and the definition of the discount factor $\b$ that the optimal $\a^*$ and $\b^*$ converge to
  \begin{align} \label{eq:convintensity}
     \alpha^{j,*}_t \rightarrow 0 \,\mbox{ for }~j\neq i, \qquad \sum_{j\neq i} \a^{j,*}_t \rightarrow \a^*_t:= \int a^*(y-Y^i_t) p_t(dy),
  \end{align}
  and $\beta^{n,*}_t \to \beta^*_t := e^{-\int_0^t \a^*_s \mathrm{d}s}$. In particular, at time $t$ the agent has $\a^*_t:= \int a^*(y-Y^i_t) p_t(dy)$ as intensity to leave the current principal, so his decision of switching is clearly influenced by the continuation value of the contact of the current principal  $Y_t$, and  by the distribution of the continuation values of the contracts of all principals $p_t$.

\begin{remark}  \label{rem:mf}
\begin{itemize}
   \item To be fair, the form of the limit equation \eqref{eq:limitY} does not entirely follow the intuition. 
         Note that we removed the $e_i$ in front of the term $Z^i_t dt$ and replaced the stochastic integrator $dW_t$ by $dW^i_t$.  
         At this stage, let us admit that once in the drift of the BSDE there is no longer dependence on other $Y^j$ ($j\neq i$), the system would be decoupled, i.e., the equation of $Y^i$ no longer needs the information of $W^j$ ($j\neq i$), and it leads to the limit form \eqref{eq:limitY}. 
         We will justify the mean field formulation in Section \ref{sec:approximation}.
   \item The first observation in \eqref{eq:convintensity} implies that in the limit case once the agent leaves the company $i$, he will have null probability to come back due to the fact that we have a continuum of players. 
         Therefore, the principal should be only interested in the optimization for the regime where the agent works for her. 
         Remember that in the $n$-principal problem, the time-inconsistency is due to the optimizations in the different regimes. 
         Therefore, it is reasonable that the mean field formulation would bypass this difficulty. 
\end{itemize}
\end{remark}

Recall the principal's problem \eqref{eq:principalfirst}. It follows from the observation in Remark \ref{rem:mf} that 
  \begin{align*}
    V^{P,i}_0 &= \begin{cases}
                    \dbE^{\dbP^{\a^*}}[X^i_0 + W_T] = X^i_0, & \mbox{as }\,\, I_0\neq i, \\
                    \displaystyle\max_{(\xi^i, \theta^i)} \dbE^{\dbP^{\a^*}}\left[X^i_\t 1_{\{\t \le T\}} + \big(X^i_T- U(\xi^i)\big)\mathbf{1}_{\{\t >T\}} -\int_0^{\t\we T} U(\theta^i_u) du\right],& \mbox{as }\,\, I_0=i.
                 \end{cases} 
  \end{align*}
Therefore, in the upcoming mean field game, we should only keep the nontrivial regime ($I_0=i$), i.e., denote
  \begin{align}
    V^{P,i}_0 := &\, \max_{(\xi^i, \theta^i)} \dbE^{\dbP^{\a^*}}\left[W^i_\t 1_{\{\t \le T\}} + \big(W^i_T- U(\xi^i)\big)\mathbf{1}_{\{\t >T\}} -\int_0^{\t\we T} \big( U(\theta^i_u) - 1\big) du\right] \notag \\
               = &\, \max_{(\xi^i, \theta^i)} \dbE^{\dbP^{\a^*}}\left[ \int_0^T  \beta^*_u \big(\alpha^*_u W^i_u - U(\theta^i_u)+1\big) du + \beta^*_T \big(W^i_T- U(\xi^i) \big) \right].\label{eq:mfprincipal1}
  \end{align}

Though the rigorous definition of the mean field game will be introduced in the next section, we are ready to give a simple description of the mean field equilibrium we are searching for. 
In the mean field formulation, we remove the superscript $i$, and instead we use the notations $\ol W, \ol Y,\ol Z$. 
 Following the dynamic programming approach in \cite{Sannikov08, CPT18}, given $(p_t)_{t\in[0,T]}$, consider the contracts in the form:
  $$ \xi \in \Xi(p) = \left\{\ol Y_T: \, \ol Y_T= \ol Y_0  - \int_0^T\left(\int c^*(y- \overline{Y}_u) p_u(dy) + \theta_u\right)du +  \int_0^T \ol Z_ud\ol W_u\right\}. $$
It follows from \eqref{eq:mfprincipal1} that each principal faces the optimization: 
\begin{equation}\label{eq:MFPrincipalProblem}
V^{P}_0(p) = \sup_{(\ol Y_0, \ol Z, \theta)}\dbE^{\dbP^{\a^*}}\left[ \int_0^T  \beta^*_u \big(\alpha^*_u \ol W_u - U(\theta_u)+1\big) du + \beta^*_T \big(\ol W_T- U( \ol Y_T)\big) \right].
\end{equation}
The maximizer $\big(\ol Y_0^*, \ol Z^*, \theta^*\big)$ would define a  process $\ol Y^*$. The law $(p_t)_{t\in [0,T]}$ is a mean field equilibrium, if $p_t = \cL\big(\ol Y^*_t\big)$ for all $t\in [0,T]$.

\subsection{Mean field game: existence}

Our rigorous definition of the mean field game and the argument to prove its existence  rely largely on the framework in Lacker \cite{Lacker15}. 

In this part of the paper, we denote the  canonical space by $ \overline{\Omega} = C([0,T], \mathbb{R}^2) $, the canonical process by $\Th := (\ol W, \ol Y) $ and the canonical filtration by $ \overline{\mathbb{F}} $.  
Given $p\in \cP(\ol \Om)$, define $p_t : = p\circ \overline{Y}^{-1}_t \in \cP(\dbR)$ and
  \begin{align*}
    \cW(p):=\Big\{\dbP^{\l, \eta}\in \cP(\ol\Om):&~~\mbox{$\ol W$ is a $\dbP^{\l,\eta}$-Brownian motion, and }\dbP^{\l,\eta}\mbox{-a.s. }\,\\\
			&~~ d \ol Y_t = -\int \big( c^*(y-\ol Y_t) p_t(dy) + \theta_t\big) dt + \eta_t d\ol W_t,~~  \mathcal{L}(\ol Y_0)=\l, \\
			&~~ \mbox{for some}~~ \l \in \cI,~~\eta\in \cU,~~ \theta\in\dbH^2\big(\dbP^{\l,\eta}\big)
            \Big\}
  \end{align*}
  where, for technical reasons, we define 
  \begin{align*}
  	 \cI &:= \big\{\l \in \cP(\dbR):~\mbox{$\l$ with a compact support $K$ in $[R,\infty)$}\big\}  \\
  	 \cU &:= \big\{\eta~ \ol\dbF\mbox{-adapted}: ~\mbox{$\eta$ takes values in a compact set $\Sigma$ in $\dbR$}\big\}.
  \end{align*}
In other words, we will consider a mean field game in which the choice of the initial value and the volatility is constrained in compact sets. Further define
  \begin{align*}
    J(\dbP, \theta;\,p):=\dbE^{\dbP}\left[ \int_0^T  \beta^*_u \big(\alpha^*_u \ol W_u - U(\theta_u)+1\big) du + \beta^*_T \big(\ol W_T- U( \ol Y_T)\big) \right],
  \end{align*}
  where we recall $\a_t^*= \int a^*(y-\ol Y_t) p_t(dy)$ and $\b_t^*=\exp\big(-\int_0^t \a^*_s ds\big)$.

\begin{assumption} \label{assum:U}
We assume that $a^*:\dbR\rightarrow \dbR $ is Lipschitz continuous. $U$ is  convex and of $q$-polynomial growth, i.e.  there are constants $C < C'$ and $q>1$ such that
  \begin{align*}
  	C(|\theta|^q -1) \le U(\theta) \le C' (|\theta|^q +1).
  \end{align*}
\end{assumption}

\begin{theorem} \label{thm:mfg}
 Under Assumption \ref{assum:c} and \ref{assum:U}, there exists  $p\in \cP(\ol \Om)$, $\big(\widehat\l, \widehat\eta\big)\in \cI\times\cU$ and an $\ol\dbF$-adapted process $\widehat \theta$ such that
   \begin{align*}
   	 \big(\dbP^{\widehat\l,\widehat\eta}, \widehat \theta \big) \in \argmin_{\dbP\in \cW(p), \theta} J(\dbP,\theta;\,p)
   	 \q\mbox{and}\q
   	 p = \dbP^{\widehat\l,\widehat\eta}.
   \end{align*}
\end{theorem}

\begin{remark}\label{rem:mfg}
\begin{itemize}
\item This definition of mean field game via the weak formulation of stochastic control follows the same spirit as those in Carmona \& Lacker \cite{CL15} and Lacker \cite{Lacker15}.  However, here we also include the control of the initial distribution $\l$ of $\ol Y$. 

\item It is noteworthy that among the triple of control $(\l,\eta, \theta)$ of this mean field game,  $\l$ takes values of measures, that is, the principals are allowed to play a mixed strategy. 

\item If we constrain $\theta$ to take values in a compact set in $\dbR$, then given a bounded function $U$ which is convex on this compact set, the mean field game exists.
\end{itemize}
\end{remark}

We realize that Theorem \ref{thm:mfg} can be proved by essentially the same argument as in \cite{Lacker15}. 
In the rest of the section, we shall outline the strategy of the proof and refer the readers for details to \cite{Lacker15}. 

First, we shall linearize the functional $J$ using the so-called relaxed control. Denote by $\cD$ the set of measures $q$ on $[0,T]\times \Sigma\times \dbR$.  Instead of controlling via the processes $\eta$ (taking values in $\Sigma$) and $\theta$ (taking values in $\dbR$), we shall control through the measure $q$ in the relaxed formulation.  The canonical space for the relaxed control becomes $\widehat\Om :=\ol\Om\times \cD$. Denote the canonical process by $(\ol W,\ol Y, \L)$, and the canonical filtration by $\widehat\dbF$. Note that one may define a $\cP(\Sigma\times \dbR)$-valued $\widehat\dbF$-predictable process $(\L_t)_{0\le t\le T}$ such that $\L(dt, d\eta, d\theta) = \L_t (d\eta, d\theta) dt$. Here we abuse the notations using $\eta, \theta$ to represent points in $(\Sigma, \dbR)$.

Denote by $ \mathcal{C}^{\infty}_0(\dbR^2) $ denote the set of infinitely differentiable functions $ \phi:\dbR^2\mapsto\dbR $ with compact support. Define the generator $ L $ on $ \mathcal{C}^\infty_0(\dbR^2) $ by
  \begin{equation*}
      L^{p,\eta}\phi(t,x,y) = \left(1, -\int c^*(y' -y ) p_t(dy') - \theta_t \right)\nabla\phi + \frac12 \big(\pa_{xx}\phi + \eta^2 \pa_{yy}\phi + 2\eta\pa_{xy}\phi\big),
  \end{equation*}
  for $ (t,x,y,p,\eta)\in[0,T]\times\dbR^2\times\mathcal{P}(\ol\Om)\times\Sigma$. Further define
  \begin{align*}
    M_t^{p,\phi} := \phi(\ol X_t,\ol Y_t) - \int_{0}^{t} \int_{\Sigma\times \dbR} L^{p,\eta}\phi(s,\ol X_s,\ol Y_s) \L_s(d\eta)ds.
  \end{align*}

\begin{definition}
  Given $ p\in\cP(\ol\Om) $, define the set of the controlled martingale problems:
   \begin{align}   \label{eq:Lintegrable}
     \cR(p) = \Big\{\widehat\dbP\in \cP\big(\widehat\Om\big): &~~ \widehat\dbP\mbox{-a.s.} ~ \ol Y_0 \sim \l, ~\mbox{for some $\l\in \cI$}, \notag \\
                                                              &~~ \dbE^{\widehat\dbP}\left[\int_0^T |\L_t|^{q} dt\right] <\infty,  \\
                                                              &~~ M^{p,\phi}~\mbox{is a $\widehat\dbP$-martingale for each $\phi\in \mathcal{C}^{\infty}_0(\dbR^2)$} \Big\}, \notag
   \end{align}
   where $|\L_t|^q: = \int_{\Sigma\times \dbR} |(\eta,\theta)|^q \L_t(d\eta,d \theta)$.
\end{definition}

\no Further, in the relaxed formulation the object function of the principals reads:
  \begin{align*}
    \hat J(\widehat\dbP;\,p) := \dbE^{\widehat\dbP}\left[ \int_0^T \int_{\dbR}\beta^*_u \big( \alpha^*_u \ol W_u - U(\theta)+1\big) \L_u(d \theta) du + \beta^*_T \big(\ol W_T- U( \ol Y_T)\big) \right],
  \end{align*}
  and define the set of the optimal control:
  \begin{align*}
    \cR^*(p) := \argmax_{\widehat\dbP\in \cR(p)} \hat J\big(\widehat\dbP;\,p\big).
  \end{align*}
 We say $\widehat\dbP\in \cP(\widehat\Om)$ is a relaxed mean field game if $\widehat\dbP\in \cR^*\big(\widehat\dbP\circ (\ol X, \ol Y)^{-1}\big)$. 
\begin{theorem}
 Under Assumption \ref{assum:c} and \ref{assum:U}, there exists a relaxed mean field game.
\end{theorem}

\begin{proof}
  Our setting slightly distinguishes from the one in \cite{Lacker15}, because we allow to control the initial law $\l$ of $\ol Y$. 
  However, since we constrain the choice of $\l$ among the distribution in $\cP(K)$ where $K$ is a compact set in $\dbR$, we are still able to prove the tightness of $\cR(p)$ (in the case $\L$ is truncated), and the rest of the argument would follow the same lines in \cite[Section 4 and 5]{Lacker15}.
\end{proof}

It remains to construct a strict mean field game (as in Theorem \ref{thm:mfg}) based on a relaxed one. Again we can follow the classical argument. 
We shall only provide the sketch of the proof, for more details we refer the readers to \cite[Proof of Theorem 3.7]{Lacker15}.

\begin{proof}[Proof of Theorem \ref{thm:mfg}]
  Let $\widehat\dbP$ be a relaxed mean field game. 
  First, by the same argument as in the proof of \cite[Proof of Theorem 3.7]{Lacker15}, we may find an $\ol\dbF$-adapted process $\widehat q:[0,T]\times \ol\Om \rightarrow \cP(\Sigma\times\dbR)$ such that 
    \begin{equation}  \label{eq:qhat}
       \widehat q(t, \ol W,\ol Y) = \dbE^{\widehat\dbP}[\L_t | \ol\cF_t], \q \widehat\dbP\mbox{-a.s.}\q t\in [0,T].
    \end{equation}
  Further, for each $(t,w,y,p)\in[0,T]\times\overline{\Omega}\times\cP(\ol\Om)$, since $\a^*, \b^*$ are always nonnegative and function $U$ is convex, the following subset
    \begin{equation*}
       K(t,\theta,y,p):=\left\{\left(\eta^2,\th,l\right):~ \eta\in\Sigma, ~ l\leq f(t,x,y,p,\theta):=\beta^*_t(y,p) \big(\alpha^*_t(y,p) \theta_t -U(\th)\big)\right\} 
    \end{equation*}
    is convex. 
  This verifies the ``convexity'' assumption in \cite[Assumption (Convex)]{Lacker15}. 
  Therefore, using the measurable selection result in \cite[Lemma 3.1]{HL90}, there exist $\ol\dbF$-adapted processes $\widehat\eta, \widehat \theta$ and $\hat l\ge 0$ such that
    \begin{align*}
      \int \big(\eta^2, f(t,x,y,p, \theta)\big)\hat q(t, x, y)(d\eta, d\theta)  = \Big(\widehat\eta^2(t,x,y), f\big(t,x,y,p, \widehat \theta(t,x,y)\big) - \hat l(t,x,y)\Big).
    \end{align*}
  Further, it would be easy to verify that $p:=\widehat\dbP\circ (\ol W, \ol Y)^{-1} =\dbP^{\l, \widehat\eta}$ as well as  $J(\dbP^{\l, \widehat\eta}, \widehat \theta;\,p) \ge \hat J(\widehat\dbP;\,p)$.
  Therefore, we find a mean field game in sense of Theorem \ref{thm:mfg}.
\end{proof}

From the sketch of proof, we may observe the following.

\begin{corollary}  \label{coro:wintegrable}
  Under Assumption \ref{assum:c} and \ref{assum:U}, there exists a mean field game such that 
   \begin{align*}
     \dbE^{\dbP^{\widehat\l,\widehat\eta}}\left[\int_0^T |\widehat \theta_u|^q du \right] <\infty.
   \end{align*}
\end{corollary}

\begin{proof}
 Note that in the proof of Theorem \ref{thm:mfg}, the $\widehat \theta$ we constructed satisfies
   \begin{align*}
     C(|\widehat \theta_t|^q -1) \le U(\widehat \theta_t) \le \int_{\dbR} U(\theta) \widehat q_t(d\theta),
   \end{align*}
  where $\widehat q$ is defined in \eqref{eq:qhat}. 
 Therefore, 
   \begin{align*}
     C\dbE^{\dbP^{\widehat\l,\widehat\eta}}\left[\int_0^T |\widehat \theta_u|^q du \right] 
      & \le C+ \dbE^{\widehat\dbP}\left[ \int_0^T \int_\dbR U(\theta) \L_u(d\theta) du\right]  \\
      & \le C+  \dbE^{\widehat\dbP}\left[ \int_0^T \int_\dbR C'(|\theta|^q+1)\L_u(d\theta) du\right].
   \end{align*}
 Then, the desired result follows from \eqref{eq:Lintegrable}.
\end{proof}

\subsection{Mean field game: approximation} \label{sec:approximation}

In this section, we shall justify our mean field formulation, that is, answer the question why the principals would apply the mean field game we studied in the previous section. Unlike the classical cases, recall that in our $n$-player problem the stochastic control problem is time-inconsistent. Therefore, it would be difficult to verify whether a contract provided by the mean field game would be $\e$-optimal for the $n$-player game. Instead, we shall verify that using a contract provided by the mean field game, the principals would receive, in a $n$-player game,  a utility $\e$-close to the value function calculated in the mean field game. 

\subsubsection{Agent problem: backward propagation of chaos} \label{sec:backwardchaos}

Here we shall analyse the agent's behaviour, once he is given $n$ contracts provided by the mean field game. 
Note that the mean field contract has the form: 
 \begin{align} \label{eq:MFBSDE}
   \xi = \ol Y_0 - \int_0^T\left(\int_{\dbR} c^*(y-\ol Y_u)p_u^*(dy) +\theta_u \right)du + \int_0^T \eta_u d \ol W_u. 
 \end{align}
In this section, we shall assume that $\xi$ is $\cF^{\ol W, \ol Y_0}_T$-measurable and $\theta$ is $\cF_t^{\ol W, \ol Y_0}$-adapted. 
In this case, we may write $\xi(\ol W, \ol Y_0)$ and $\theta(\ol W, \ol Y_0)$.  
Now, recall the canonical space $\Om$ in the setting of $n$-principal problem, and the canonical process $W=\{W^i\}_{i\in \dbI_n}$ representing the outputs of all principals. 
Further, let $\{(Y^{*,i}, \xi^i, \theta^i, \eta^i)\}_{i\in \dbI_n}$ be the $n$ independent copies of $(\ol Y, \xi, \theta, \eta)$
such that
 \begin{align}  \label{eq:MFBSDEi}
    Y_t^{*,i} &= \xi^{i} + \int_t^T\left(\int_{\dbR} c^*(y-Y^{*,i}_u)p_u^*(dy) + \theta^i_u \right)du - \int_t^T \eta^i_u d  W^i_u \notag \\
              &= \xi^{i} + \int_t^T\left(\int_{\dbR} c^*(y-Y^{*,i}_u)p_u^*(dy) + \theta^i_u \right)du - \int_t^TZ_u^{*,i}\cd dW_u,
 \end{align}
  with ${Z}^{*,i}:=  e_i \eta^i = \big(0,\cdots,0, \eta^i, 0, \cdots, 0\big)^\top$. 

In this section, we shall simply use the notation $\dbE$ instead of $\dbE^{\dbP^{\widehat\l,\widehat\eta}}$. The following estimate follows directly from Corollary \ref{coro:wintegrable}.

\begin{lemma}\label{lem:estimateYstar}
Assume $q \ge  2$ in Assumption \ref{assum:U}. Then we have 
  \begin{align*}
      \dbE\left[\sup_{0\leq t\leq T}\big|Y_t^{*,i}\big|^2\right] <\infty. 
  \end{align*}
In particular $\dbE[|\xi^i|^2]<\infty$.
\end{lemma}
   
Given such contracts $\{(\xi^i, \theta^i)\}_{i\in \dbI_n}$, the agent would solve the system of BSDE \eqref{eq:BSDEnagent}, namely, 
  \begin{align}  \label{eq:Yni}
     Y_t^{n,i} &= \xi^{i} + \int_t^T\Bigg(\tfrac{1}{n-1}\sum_{\substack{j=1\\j\neq i}}^nc^*(Y^{n,j}_u-Y^{n,i}_u) + \theta^i_u\Bigg)du - \int_t^T Z^{n,i}_u\cdot dW_u \notag \\
               &= \xi^{i} + \int_t^T\bigg(\tfrac{n}{n-1}\int_{\dbR}c^*(y-Y^{n,i}_u)p^n_u(dy) - \tfrac{c^*(0)}{n-1} + \theta^i_u\bigg)du - \int_t^T Z^{n,i}_u\cdot dW_u, 
  \end{align}
 where $p^{n}$ is the empirical measure $ \frac{1}{n}\sum_{j=1}^n \d_{Y^{n,j}}$ and $p^n_u:= \frac{1}{n}\sum_{j=1}^n \d_{Y^{n,j}_u}$. 
Define $\Delta Y^i:=Y^{n,i}-Y^{*,i}$, $\Delta Z^i:=Z^{n,i}-Z^{*,i}$. 
Let $\cC:=C([0,T];\dbR)$.
Denote the square of the Wasserstein-$2$ distance on $\cP^2(\cC)$ by
  \begin{align*}
     d^2_t(\mu,\nu) :=\inf_{\pi\in\Pi(\mu,\nu)}\int_{\cC\times\cC}\sup_{t\leq u\leq T}|x_u-y_u|^2\pi(dx,dy). 
  \end{align*}
Here is the main result concerning the agent's problem.

\begin{proposition} \label{prop:backpropa}
Assume that Assumption \ref{assum:c} and \ref{assum:U} hold true for some $q \ge 2$. 
Then we have
  \begin{align}  \label{eq:backpropa}
    \lim_{n\to\infty}\dbE\left[d^2_0\left(p^n,p^*\right)\right] = 0. 
  \end{align}
\end{proposition}

\begin{remark}
\begin{itemize}
  \item There is a certain similarity between the previous result and the one in Buckdahn, Djehiche, Li, and Peng \cite{BDLP09} where the authors also study the convergence from a $n$-player BSDE to the mean field limit equation in the form of \eqref{eq:MFBSDEi}. 
        To our understanding of their paper, the solution to the $n$-player BSDE there is a fixed point of the following map:
          \begin{align*}
            p\mapsto Y(p) \mapsto \frac1n \sum_{i=1}^n \d_{Y^i(p)},
          \end{align*}
          where $Y(p)$ is the solution of the BSDE
          \begin{align}  \label{BSDEpeng}
            Y_t = \xi + \int_t^T\left(\int_{\dbR} c^*(y-Y_u)p_u(dy) + \theta_u \right)du - \int_t^T Z_ud\ol W_u,
          \end{align}
          and $\{Y^i(p)\}_{i\in \dbI_n}$ are the independent copies of $Y(p)$. 
        In other words, their formulation of the $n$-player problem is via the open loop while ours is via the closed loop. 

  \item As we will show, the technique involved to prove Proposition \ref{prop:backpropa} is a combination of the BSDE estimates and the argument for proving the propagation of chaos. 
        That is why we would name this section the backward propagation of chaos. When we revise the paper, we note that this topic has already attracted some following research, see \cite{LaT19}. 

  \item Indeed the result of Proposition \ref{prop:backpropa} and the upcoming analysis hold valid for the broader class of BSDE systems in the form of
         \begin{align*}
            Y^{n,i}_t = \xi^i +\int_t^T F\big(t, Y_t^{n,i}, Z_t^{n,i},  p^n_t\big)dt -Z^{n,i}_t dW_t,\q i\in \dbI_n,
         \end{align*}
         where $F$ is Lipschitz continuous in $(y,z,p)$. 
        Apparently,  this result could have an independent interest, once one wants to look into the asymptotic behaviour of solutions to such BSDE systems.
\end{itemize}
\end{remark}

Before proving this main result, we first obtain the following estimates through some classical BSDE arguments. 

\begin{lemma}
We have 
   \begin{align}
     \dbE\left[\sup_{t\leq s\leq T}|\Delta Y^i_s|^2\right] &\leq C\left(\dbE\left[\int_t^Td^2_u(p^{n},p^*)du\right]+ \frac1{n^2}\right), \label{eq:supDeltaY<Cd2} \\
     \dbE\left[\int_t^T|\Delta Z^i_u|^2du\right]  &\leq  C \left(\dbE\left[\int_t^Td^2_u(p^{n},p^*)du\right] + \frac{1}{n^2}\right). \label{eq:finalestimateZ}
   \end{align}
\end{lemma}

\begin{proof}
Comparing \eqref{eq:MFBSDEi} and \eqref{eq:Yni} we obtain
\begin{equation} \label{eq:deltaY<}
  \begin{aligned}
    |\Delta Y^i_t| &\leq \frac{n}{n-1}\int_t^T\bigg|\int_{\dbR}c^*(y-Y_u^{n,i})p_u^{n}(dy)-\int_{\dbR}c^*(y-Y_u^{*,i})p^*_u(dy)\bigg|du + \frac{|c^*(0)|}{n-1}T \\
                   & \qquad + \frac{1}{n-1}\int_t^T\int_{\dbR}c^*(y-Y_u^{*,i})p_u^*(dy)du  + \bigg|\int_t^T\Delta Z^i_u\cdot dW_u\bigg|.
  \end{aligned}
\end{equation}
Using the Kontorovich duality and Lipschitz continuity of $c^*$, we have  
\begin{align} \label{eq:triangleIneqMeasures}
   & \bigg|\int_{\dbR}c^*(y-Y_u^{n,i})p_u^{n}(dy)-\int_{\dbR}c^*(y-Y_u^{*,i})p^*_u(dy)\bigg| \nonumber \\
   & \quad \leq \bigg|\int_{\dbR}c^*(y-Y_u^{n,i})p_u^{n}(dy)-\int_{\dbR}c^*(y-Y_u^{n,i})p^*_u(dy)\bigg| \nonumber \\ 
   & \hspace{10mm} + \bigg|\int_{\dbR}c^*(y-Y_u^{n,i})p^*_u(dy)-\int_{\dbR}c^*(y-Y_u^*)p^*_u(dy)\bigg| \nonumber \\
   & \quad \leq \cW_1(p^{n}_u,p^*_u) + L|\Delta Y^i_u| \nonumber \\
   & \quad \leq \cW_2(p^{n}_u,p^*_u) + L|\Delta Y^i_u|. 
\end{align}
Squaring and taking supremum and expectation on both sides of \eqref{eq:deltaY<}, and using Jensen and BDG-inequality, we obtain 
 \begin{align} \label{eq:supDY1}
   \dbE\left[\sup_{t\leq s\leq T}|\Delta Y^i_s|^2\right] 
     &\leq C_{T}\dbE\left[\int_t^T\cW^2_2(p_u^{n},p_u^*)du\right] + C_{L,T}\dbE\left[\int_t^T|\Delta Y^i_u|^2du\right] + \frac{C_T}{n^2}|c^*(0)|^2  \nonumber \\
     &\qquad + C_{T}\dbE\left[\sup_{t\leq s\leq T}\left|\int_t^s\Delta Z^i_u\cdot dW_u\right|^2\right] \nonumber \\
     &\qquad + \frac{C_T}{n^2}\dbE\left[\int_t^T\left|\int_{\dbR}c^*(y-Y^{*,i})p_u^*(dy)\right|^2du\right]\nonumber \\
     &\leq C_{T}\dbE\left[\int_t^T\cW^2_2(p_u^{n},p_u^*)du\right] + C_{L,T}\dbE\left[\int_t^T|\Delta Y^i_u|^2du\right] + \frac{C_T}{n^2}|c^*(0)|^2  \\
     &\qquad + C_{T,BDG}\dbE\left[\int_t^T|\Delta Z^i_u|^2du\right] + \frac{C_T}{n^2}\dbE\left[\int_t^T\left|\int_{\dbR}c^*(y-Y^{*,i})p_u^*(dy)\right|^2du\right]. \nonumber
 \end{align}
Further we shall estimate $\dbE\left[\int_t^T|\Delta Z^i_u|^2du\right] $. 
 By It\^o's formula, 
  \begin{align*}
    |\Delta Y^i_t|^2 + \int_t^T|\Delta Z^i_u|^2du 
    &= \frac{2n}{n-1}\int_t^T\Delta Y^i_u\left(\int_{\dbR}c^*(y-Y_u^{n,i})p_u^{n}(dy)-\int_{\dbR}c^*(y-Y_u^{*,i})p^*_u(dy)\right)du \\
    & \qquad -\frac{2c^*(0)}{n-1}\int_t^T\Delta Y^i_udu - 2\int_t^T\Delta Y^i_u\Delta Z^i_u\cdot dW_u \\  
    & \qquad + \frac{2}{n-1}\int_t^T\Delta Y^i_u\int_{\dbR}c^*(y-Y^{*,i}_u)p^*_u(dy)du.
  \end{align*}
Together with \eqref{eq:triangleIneqMeasures} and Young's inequality, we obtain
 \begin{align} \label{eq:estimatedeltaZ} 
   \dbE\left[ \int_t^T|\Delta Z^i_u|^2du\right] 
     & \leq 2\dbE\left[\int_t^T|\Delta Y^i_u|\left(\cW_2(p_u^{n},p_u^*) + C_L|\Delta Y_u^i|\right)du\right] + \frac{c^*(0)^2}{(n-1)^2}T \notag \\
     & \qquad + 2\dbE\left[\int_t^T|\Delta Y^i_u|^2du\right] + \frac{1}{(n-1)^2} \dbE\left[\int_t^T\left|\int_{\dbR}c^*(y-Y^{*,i}_u)p^*_u(dy)\right|^2du\right] \notag \\
     & \leq (3+2C_L)\dbE\left[\int_t^T|\Delta Y^i_u|^2du\right] + \dbE\left[\int_t^T\cW^2_2(p_u^{n},p_u^*)du\right]  \notag \\
     & \qquad + \frac{c^*(0)^2}{(n-1)^2}T + \frac{1}{(n-1)^2} \dbE\left[\int_t^T\left|\int_{\dbR}c^*(y-Y^{*,i}_u)p^*_u(dy)\right|^2du\right].
 \end{align}
 We now estimate the last term above, 
  \begin{align*}
    \dbE\left[\int_t^T\left|\int_{\dbR}c^*(y-Y^{*,i}_u)p^*_u(dy)\right|^2du\right] 
       &\leq \dbE\left[\int_t^T\left(\int_\dbR\left(L|y| + L|Y^{*,i}_u| + |c^*(0)|\right)p_u^*(dy)\right)^2du\right] \\
       & \leq \dbE\left[\int_t^T\left(L\dbE[|Y^{*,i}_u|]+L|Y_u^{*,i}|+|c^*(0)|\right)^2du\right]  \\
       & \leq 6L^2T\dbE\left[\sup_{0\leq u\leq T}|Y^{*,i}_u|^2\right] + 3T|c^*(0)|^2, 
  \end{align*}
  which is bounded from above by a constant $C_0$, independent of $t$ and $i$, by the a priori estimate for  $Y^{*,i}$ in Lemma \ref{lem:estimateYstar}. 
  Together with \eqref{eq:supDY1} we obtain
  \begin{align*}
    \dbE\left[\sup_{t\leq s\leq T}|\Delta Y^i_s|^2\right] 
      &\leq (C_{T}+C_{T,BDG})\dbE\left[\int_t^T\cW^2_2(p_u^{n},p_u^*)du\right] \\ 
      &\qquad + \big(C_{L,T}+(2C_L+3)C_{T,BDG}\big)\int_t^T\dbE\bigg[\sup_{u\leq s\leq T}\big|\Delta Y^i_s\big|^2\bigg]du \\
      &\qquad + \frac1{n^2}(C_T + 4C_{T,BDG}T)c^*(0)^2 + \frac1{n^2}(C_T+4C_{T,BDG})C_0.
  \end{align*}  
 Applying Gr\"onwall inequality, we get 
  \begin{align*}
    \dbE\left[\sup_{t\leq s\leq T}|\Delta Y^i_s|^2\right] \leq C\left(\dbE\left[\int_t^T\cW^2_2(p_u^{n},p_u^*)du\right]+\frac1{n^2}\right) \leq C\left(\dbE\left[\int_t^Td^2_u(p^{n},p^*)du\right]+\frac1{n^2}\right),
  \end{align*}
  that is, the estimate \eqref{eq:supDeltaY<Cd2},  for some constant depending on $T$, $L$, $C_0$ and the constant from BDG inequality. Finally, the estimate \eqref{eq:finalestimateZ} follows from \eqref{eq:estimatedeltaZ} and \eqref{eq:supDeltaY<Cd2}. 
\end{proof}

\ms

\begin{proof}[Proof of Proposition \ref{prop:backpropa}]
	Define the empirical measure,
	   $$ \nu^n:= \frac{1}{n}\sum_{i=1}^n\delta_{Y^{*,i}}. $$ 
	The empirical measure $\frac{1}{n}\sum_{i=1}^n\delta_{(Y^{n,i},Y^{*,i})}$ is a coupling of the empirical measures $p^{n}$ and $\nu^n$,  so 
	   $$ d^2_t(p^{n},\nu^n) \leq \frac{1}{n}\sum_{i=1}^n\sup_{t\leq u\leq T}\big|\Delta Y^{i}_u\big|, \quad a.s. $$  
	Together with \eqref{eq:supDeltaY<Cd2}, we obtain that 
	   $$ \dbE\left[d^2_t(p^{n},\nu^n)\right] \leq C\left(\dbE\left[\int_t^Td_u^2(p^n,p^*)du\right]+\frac1{n^2}\right). $$
	Apply the triangle inequality and the previous inequality to obtain 
	  \begin{align*}
	   \dbE\big[d^2_t(p^n,p^*)\big] 
	     &\leq 2\dbE\big[d^2_t(p^n,\nu^n)\big] + 2\dbE\big[d_t^2(\nu^n,p^*)\big]  \\ 
	     &\leq 2C\dbE\left[\int_t^Td^2_u(p^n,p^*)du\right] + 2\dbE\big[d^2_t(\nu^n,p^*)\big]+\frac{2C}{n^2}. 
	  \end{align*}
	Using Gr\"onwall's inequality  we obtain 
	  \begin{align*}
	    \dbE\big[d^2_0(p^n,p^*)\big] \leq 2e^{2CT}\left(\dbE\big[d^2_0(\nu^n,p^*)\big]+\frac C{n^2}\right).
	  \end{align*}
	Each $\nu^n$ is the empirical measures of i.i.d.~samples from the law $p^*$, \eqref{eq:backpropa} follows from the law of large numbers. 
\end{proof}

\subsubsection{Principals' problem}

For technical reasons, in this section we would consider the case where $\theta$ takes values in a compact set in $\dbR$ and $U$ is convex and bounded on this set.
Recall the third point in Remark \ref{rem:mfg}, under the above setting, we still have existence of the mean field games.  

\begin{proposition}
Let all the $n$ principals offer the contract provided by a mean field game, for principal $i$ the reward becomes
  \begin{align} \label{eq:valuenplayer}
    V^{n,i}_0 &= \int_K\l(dy_0)\dbE^{\dbP^{\alpha^*}}\left[X^i_T - U(\xi^i) \mathbf{1}_{\{I_T=i\}} - \int_0^T U(\theta^i_u) \mathbf{1}_{\{I_u=i\} }du \right],
  \end{align}
  where $\a^*$ is the optimal intensity of the agent satisfying \eqref{eq:optimalintensity}. 
Then, as $n\rightarrow\infty$, $V^{n,i}_0$ converges to the value of the mean field game. 
\end{proposition}
 
 \begin{proof}
The dynamic version of \eqref{eq:valuenplayer} reads
  \begin{equation*}
     V^{n,i}_t = \dbE_t^{\dbP^{\alpha^*}}\left[X^i_T - U(\xi^i) \mathbf{1}_{\{I_T=i\}}  -\int_t^T U(\theta^i_u) \mathbf{1}_{\{I_u=i\} }du \right], \q\mbox{for $0<t\le T$}.
  \end{equation*}
Since $U$ is bounded
  \begin{align} \label{eq:Vniintegrable}
    \dbE^{\dbP^{\a^*}}\bigg[\sup_{0\le t\le T}\big|V^{n,i}_t\big|^2\bigg] <\infty.
  \end{align}
We denote $V^{n,i,0}_t := V^{n,i}_t\big|_{I_t \neq i}$ and $V^{n,i,1}_t := V^{n,i}_t\big|_{I_t = i}$ on different regimes. 
 Recall $\t^n_t=\inf\{s\ge t: I_s \neq I_t\}$ and define $\ol \t^n_t :=\inf\{s\ge t: I_s=i\}$.
 By the tower property of the conditional expectation, we have
   \begin{align}  \label{eq:Vn0}
     \begin{cases}
       \displaystyle V^{n,i,1}_t =\dbE_t^{\dbP^{\a^*}}\left[V^{n,i,0}_{\t^n_t} \mathbf{1}_{\{\t^n_t \le T\}} + \big(X^i_T- U(\xi^i)\big)\mathbf{1}_{\{\t^n_t >T\}} -\int_t^{\t^n_t \we T} U(\theta^i_u) du\right]\\ 
        \vspace{-3mm} \\
        \displaystyle V^{n,i,0}_t = \dbE_t ^{\dbP^{\a^*}}\left[ V^{n,i,1}_{\ol\t^n_t} \mathbf{1}_{\{\ol\t^n_t\le T\}} +X^i_T\mathbf{1}_{\{\ol\t^n_t >T\}} \right] 
         = \dbE_t^{\dbP^{\a^*}}\left[\big(V^{n,i,1}_{\ol\t^n_t}-X^i_T\big)\mathbf{1}_{\{\ol\t^n_t \le T\}}\right] +X^i_t.
     \end{cases}
   \end{align}
It follows from \eqref{eq:optimalintensity} that $\lim\limits_{n\rightarrow\infty}\dbP^{\a^*}[\ol \t^n_t \le T] =0$. Together with \eqref{eq:Vniintegrable}, we obtain
   \begin{align}  \label{eq:Vbadregime}
   	  \lim_{n\rightarrow \infty} \dbE^{\dbP^{\a^*}}\Big[\big|V^{n,i,0}_t - X^i_t\big|^2\Big]=0.
   \end{align} 
Define
   \begin{align*}
      \widetilde V_0^{n,i,1}=\int_K\l(dy_0)\dbE^{\dbP^{\a^*}}\left[\int_0^T\beta^n_u \big( \alpha^n_u  W^i_u - U(\theta^i_u)+1\big) \mathrm{d}u + \beta^n_T \big(W^i_T- U(\xi^i) \big) \right],
   \end{align*}
 where $\a^n_t = \frac{1}{n-1} \sum_{j\neq i }a^*(Y^{n,j}_t-Y^{n,i}_t)$ and $\b^n_t := \exp\big(-\int_0^t \a^n_u du\big)$.
Note that $\a^*,\b^*$ are bounded. 
By \eqref{eq:Vbadregime}, we have
   \begin{align*}
       \lim_{n\rightarrow\infty}\big|\widetilde V_0^{n,i,1}- V_0^{n,i,1}\big| = 0.
   \end{align*}
Further, on the regime $I = i$, the expectation $\int_K\widehat\l(dy_0)\dbE^{\a^*}[\cd]$ coincides with $\dbE$(=$\dbE^{\widehat\l,\widehat\eta}$), so
   \begin{align*}
      \widetilde V_0^{n,i,1}
         =\dbE\left[\int_0^T  \beta^n_u \big( \alpha^n_u W^i_u - U(\theta^i_u) +1\big)du + \beta^n_T \big(W^i_T- U(\xi^i) \big) \right].
   \end{align*}
Since the function $a^*$ is bounded and Lipschitz continuous, we have
  \begin{align*}
     |\b^n_t-\b^*_t|^2 + |\a^n_t\b^n_t - \a^*_t\b^*_t|^2 \le C\Big(d^2_0(p^n, p^*) + \big|Y^{n,i}_t - Y^{*,i}_t\big|^2\Big).
  \end{align*}
Finally, the convergence result \eqref{eq:backpropa} and the estimate \eqref{eq:supDeltaY<Cd2} imply that
  \begin{align*}
    \lim_{n\rightarrow \infty}\widetilde V_0^{n,i,1}
      =\dbE\left[\int_0^T \beta^*_u \big( \alpha^*_u W^i_u - U(\theta^i_u)+1\big) du + \beta^*_T \big(W^i_T- U(Y^{*,i}_T) \big)\right],
  \end{align*}
and the latter is the value of the mean field game. 
\end{proof}

\bibliography{PAProblem} 
 
\bibliographystyle{abbrv}  

\end{document}